\documentclass[10pt]{amsart}
\usepackage{rotating}
\usepackage{mathcomp, amscd}
\usepackage{times}
\usepackage{amsmath, amssymb, amsthm, fancyhdr, verbatim, graphicx}
\usepackage{enumerate}
\usepackage{enumitem}
\usepackage[all]{xy}
\usepackage{accents}
\usepackage{mathtools}
\usepackage{bbm} 
\usepackage{extarrows}
\usepackage{stmaryrd}
\usepackage{array}
\usepackage{float}
\usepackage[unicode=true]{hyperref}
\usepackage{framed}
\usepackage{url}
\usepackage{scalerel}
\usepackage{booktabs}

\usepackage{bbding}
\usepackage{pifont}

\usepackage[usenames,dvipsnames]{xcolor}
\usepackage{mathrsfs}
\usepackage{caption}
\usepackage{tikz-cd}
\usepackage{color}

\definecolor{byzantium}{rgb}{0.44, 0.16, 0.39}

\definecolor{GOOD}{HTML}{2C5C88}

\definecolor{bluedefrance}{rgb}{0.19, 0.55, 0.91}

\definecolor{OxfordBlue}{RGB}{0, 33, 71}
\definecolor{JHUBlue}{RGB}{0, 127, 174}
\definecolor{QueensRed}{RGB}{171, 35, 40}
\definecolor{QueensGold}{RGB}{255, 199, 44}

\definecolor{ImpressionistBlue}{HTML}{5B7F95}
\definecolor{WaterLilyGreen}{HTML}{7D956D}
\definecolor{SunsetPeach}{HTML}{F4A988}
\definecolor{LilacMist}{HTML}{C3A2C8}
\definecolor{SunflowerGold}{HTML}{E3A048}
\definecolor{WimP}{HTML}{4C2B69}
\definecolor{CrownAmethyst}{HTML}{3B2453}

\definecolor{SaxonGold}{HTML}{FFD700}
\definecolor{SaxonGreen}{HTML}{008B45}
\definecolor{DresdenYellow}{HTML}{FFCC00}

\definecolor{Pistachio}{HTML}{84B067}
\definecolor{Avacado}{HTML}{568203}
\definecolor{Malachite}{HTML}{0BDA51}
\definecolor{SageGreen}{HTML}{8A9A5B}
\definecolor{Gua}{HTML}{355E3B}
\definecolor{Wim}{HTML}{005F36}
\definecolor{Froggie}{HTML}{7b908e}
\definecolor{Cerulean}{HTML}{007BA7}
\definecolor{BlueJay}{HTML}{2B547E}
\definecolor{CobaltBlue}{HTML}{0047AB}
\definecolor{ImperialSapphire}{HTML}{1E3A5F}
\definecolor{RegalAzure}{HTML}{2C5C88}
\definecolor{Sapphire1}{HTML}{172f5d}
\definecolor{Sapphire2}{HTML}{2b3dab}
\definecolor{Sapphire3}{HTML}{4555b8}
\definecolor{Sapphire4}{HTML}{5092ce}
\definecolor{DeepSeaBlue}{HTML}{123456}
\definecolor{DarkBlueGrey}{HTML}{29465B}
\definecolor{Burgundy}{HTML}{660033}
\definecolor{Cartier}{HTML}{9D2235}
\definecolor{Marigold}{HTML}{EAA221}
\definecolor{Brass}{HTML}{B5A642}

\definecolor{Roof}{HTML}{CE4B1B}
\definecolor{Sea}{HTML}{0396BF}

\definecolor{ByzPurple}{HTML}{5f3672}
\definecolor{ByzYellow}{HTML}{ba923f}
\definecolor{ByzPink}{HTML}{845a60}

\definecolor{Delft1}{RGB}{182,92,70}
\definecolor{Delft2}{RGB}{42,67,104}
\definecolor{Delft3}{RGB}{145,123,95}
\definecolor{StarlingGround}{HTML}{024E57}
\definecolor{StarlingLinen}{HTML}{C3B49D}
\definecolor{StarlingFlesh}{HTML}{AF8365}
\definecolor{StarlingRusset}{HTML}{9A5E41}
\definecolor{StarlingFig}{HTML}{576327}

\newcommand{\F}{\mathbb{F}}

\newcommand{\Z}{\mathbb{Z}}
\renewcommand{\P}{\mathbb{P}}

\newcommand{\N}{\mathbb{N}}

\newcommand{\OO}{\mathcal{O}}

\newcommand{\sL}{\mathscr{L}}

\newcommand{\cE}{\mathcal{E}}
\newcommand{\cF}{\mathcal{F}}
\newcommand{\cG}{\mathcal{G}}
\newcommand{\cH}{\mathcal{H}}
\newcommand{\cI}{\mathcal{I}}
\newcommand{\cJ}{\mathcal{J}}

\newcommand{\cL}{\mathcal{L}}
\newcommand{\cM}{\mathcal{M}}

\newcommand{\cR}{\mathcal{R}}

\newcommand{\cV}{\mathcal{V}}
\newcommand{\cW}{\mathcal{W}}

\newcommand{\Qp}{\mathbb{Q}_p}

\newcommand{\Zp}{\mathbb{Z}_{p}}

\newcommand{\Fp}{\mathbb{F}_p}

\newcommand{\Fq}{\mathbb{F}_{q}}

\newcommand{\wt}[1]{\widetilde{#1}}

\newcommand{\twomat}[4]{\begin{pmatrix}#1 & #2 \\ #3 & #4\end{pmatrix}}

\newcommand{\Gal}{\operatorname{Gal}}

\newcommand{\ul}[1]{\underline{#1}}
\newcommand{\ol}[1]{\overline{#1}}
\newcommand{\wh}[1]{\widehat{#1}}

\DeclareMathOperator{\et}{\acute{e}t}

\makeatletter
\newcommand*\bigcdot{\mathpalette\bigcdot@{.7}}
\newcommand*\bigcdot@[2]{\mathbin{\vcenter{\hbox{\scalebox{#2}{$\m@th#1\bullet$}}}}}
\makeatother

\newcommand{\modifier}[1]{}

\def\llparent{( \! ( }
\def\rrparent{) \! ) }
\newcommand{\bs}{\backslash}
\newcommand{\surj}{\twoheadrightarrow}

\newcommand{\dirlim}{\varinjlim}
\newcommand{\invlim}{\varprojlim}

\newcommand{\HHom}{\mathcal{H}\mathrm{om}}

\newcommand{\tilt}{\flat}

\newlist{renumerate}{enumerate}{3}
\setlist[renumerate]{label=(\roman*),before=\raggedright, itemsep=0mm}
\newlist{arenumerate}{enumerate}{3}
\setlist[arenumerate]{label=(\arabic*),before=\raggedright, itemsep=0mm}
\newlist{aenumerate}{enumerate}{3}
\setlist[aenumerate]{label=(\alph*),before=\raggedright, itemsep=0mm}
\setlist[enumerate]{itemsep=0mm}

\DeclareMathOperator{\GL}{GL}
\DeclareMathOperator{\SL}{SL}

\DeclareMathOperator{\coker}{coker}

\DeclareMathOperator{\Dec}{Dec}

\DeclareMathOperator{\Sp}{Sp}

\DeclareMathOperator{\Hom}{Hom}

\DeclareMathOperator{\Ind}{Ind}
\DeclareMathOperator{\im}{im\,}
\DeclareMathOperator{\rank}{rank\,}

\DeclareMathOperator{\Rep}{Rep}

\DeclareMathOperator{\LT}{LT}

\DeclareMathOperator{\Div}{Div}

\DeclareMathOperator{\id}{id}

\DeclareMathOperator{\Gr}{Gr}
\DeclareMathOperator{\Pl}{Pl}

\DeclareMathOperator{\cts}{cts}

\DeclareMathOperator{\sm}{sm}

\DeclareMathOperator{\adm}{adm}
\DeclareMathOperator{\rad}{rad}

\DeclareMathOperator{\Perf}{Perf}

\DeclareMathOperator{\Spa}{Spa}
\DeclareMathOperator{\Spd}{Spd}

\DeclareMathOperator{\Coeq}{Coeq}

\DeclareMathOperator{\bB}{\mathbf{B}}

\DeclareMathOperator{\Inj}{\mathcal{I}\mathrm{nj}}
\DeclareMathOperator{\Surj}{\mathcal{S}\mathrm{urj}}
\DeclareMathOperator{\BC}{BC}

\usetikzlibrary{calc,decorations.pathmorphing,shapes}
\newcounter{sarrow}

\tikzcdset{scale cd/.style={every label/.append style={scale=#1},
    cells={nodes={scale=#1}}}}

\newtheorem{thm}{Theorem}
\newtheorem{cor}[thm]{Corollary}
\newtheorem{prop}[thm]{Proposition}
\newtheorem{lem}[thm]{Lemma}
\newtheorem{conj}[thm]{Conjecture}
\newtheorem{quest}[thm]{Question}
\newtheorem{thmx}{Theorem}

\theoremstyle{definition}
\newtheorem{defn}[thm]{Definition}

\newtheorem{exmp}[thm]{Example}
\newtheorem{exmps}[thm]{Examples}
\newtheorem{notn}[thm]{Notation}

\newtheorem*{prob*}{Problem}
\newtheorem*{goal*}{Goal}

\theoremstyle{remark}
\newtheorem{rem}[thm]{Remark}
\newtheorem*{rem*}{Remark}
\newtheorem{rems}[thm]{Remarks}

\newcommand{\para}[1][]{%
  \refstepcounter{thm}%
  \subsection*{%
    \thethm.
    \if\relax\detokenize{#1}\relax
    \else
      \textbf{#1}
    \fi
  }%
}

\numberwithin{equation}{section}
\numberwithin{thm}{section}
\usepackage[cal=euler, scr=boondox]{mathalfa} 

\hypersetup{
    colorlinks=true,
    linkcolor=StarlingGround,
    urlcolor=StarlingFig,
    citecolor=StarlingRusset
    }

\theoremstyle{definition}

\usepackage[
backend=biber,
sorting=nyvt,
style=alphabetic,
maxalphanames=5,
minalphanames=5,
maxbibnames=10,
backref=true
]{biblatex}
\renewbibmacro{in:}{}
\DeclareFieldFormat[article]{volume}{\textbf{#1}}
\DeclareFieldFormat[article]{number}{no.~#1}
\DeclareFieldFormat[article]{year}{\mkbibparens{#1}}
\renewbibmacro*{volume+number+eid}{%
  \printfield{volume}%
  \setunit*{\addspace}
  \printfield{number}%
  \setunit{\addcomma\space}
  \printfield{eid}%
}

\DeclareMathOperator{\dec}{dec}
\DeclareMathOperator{\Ord}{Ord}

\title{Grassmannians of absolute Banach--Colmez Spaces}
\author{Zecheng Yi}
\email{\url{zechengyi0822@gmail.com}}
\address{Boston University Department of Mathematics and Statistics, 665 Commonwealth Avenue, Boston, MA, USA}

\begin{document}
\begin{abstract}
We introduce Grassmannians of absolute Banach--Colmez spaces and prove  the  $k$-Grassmannian is perfectoid if and only if the space has slope $< 1/k$. For $d=1$, this proves a conjecture of Hansen--Mann. The key ingredient is a Pl\"{u}cker morphism, which realizes the $k$-Grassmannian of a slope $d/n$ space as a Zariski closed subspace of the projectivization of a slope $kd/n$ space. We also introduce a general method to construct invariant elements in perfectoid Tate--Huber pairs, and use this to prove absolute Banach--Colmez spaces of slope $<1$ have perfectoid projectivizations. As applications, we establish  vanishing bounds for Scholze's mod $p$ Jacquet--Langlands functor on parabolic inductions, and also prove its vanishing  in the top degree on every infinite-dimensional non-supersingular irreducible representation.
\end{abstract}
\maketitle
\tableofcontents

\section*{Introduction}

Introduced by Colmez in \cite{Colmez2002-1}, Banach--Colmez (BC) spaces play a key role in $p$-adic Hodge theory. Fargues and Fontaine \cite{FF2018-1} gave them a geometric interpretation  through  the cohomology of vector bundles on the Fargues--Fontaine curve.  Le Bras \cite{LeBras2018-1} further developed this interpretation categorically, viewing BC spaces as $v$-sheaves of $\Qp$-vector spaces on $\Perf_C$ and relating them to coherent sheaves on the curve.

Allowing the base to vary without choosing an untilt gives absolute BC spaces \cite{Fargues2020-1}, which are $v$-sheaves on $\Perf_{\Fq}$.  Their projectivizations include familiar moduli spaces: for $d\geq 1$, the quotient of the nonzero sections of $\OO(d)$ by scalar multiplication is isomorphic to the space $\Div^{d}$ of effective Cartier divisors of degree $d$ \cite{Fargues2020-1}. Motivated by the spaces $W_{n,k}$ of Hansen--Mann \cite{HM2022-1}, we introduce Grassmannians of absolute BC spaces of positive slope and higher rank, extending the projectivizations studied in \cite{Fargues2020-1}\cite{FS2024-1}. The primary goal of this article is to study geometric properties of these Grassmannians, and our main result determines when they are representable by perfectoid spaces.

Let $K/\Qp$ be a finite extension with uniformizer $\pi$ and residue field $\Fq$. Let  $C$ be a complete algebraically closed extension of $K$. Write  $\bB^{\varphi^{n}=\pi^{d}}$ for the absolute BC space of rank $n$ and degree $d>0$; its slope is given by  $\mu=d/n$.  As a $v$-sheaf on $\Perf_{\Fq}$, $\bB^{\varphi^{n}=\pi^{d}}$ is defined by 
\[
	\bB^{\varphi^{n}=\pi^{d}}: S\mapsto H^{0}(X_S, \cE_{n,d}), \quad S\in \Perf_{\Fq}
\] 
where $\cE_{n,d}$ is the standard semistable vector bundle of rank $n$ and degree $d$. For $0<k<n$, let $\Inj(\OO^{k}, \cE_{n,d})$ be the $v$-sheaf on $\Perf_{\Fq}$ parametrizing stably injective homomorphisms  $\OO^{k}\hookrightarrow \cE_{n,d}$.  We define the $k$-Grassmannian of $\bB^{\varphi^{n}=\pi^{d}}$ as 
\[
	\Gr_{k}(\bB^{\varphi^{n}=\pi^{d}}):= \Inj(\OO^{k}, \cE_{n,d})/\underline{\GL_k(K)}
\] 
and write  $\Gr_{k,C}(\bB^{\varphi^{n}=\pi^{d}})$ for its base change to $\Spd C$. Our main result is the following:

\begin{thmx}[Theorem \ref{thm-dk<n} and Proposition \ref{prop-nonperf}]\label{thm-A}
	The $k$-Grassmannian  $\Gr_{k,C}(\bB^{\varphi^{n}=\pi^{d}})$ is perfectoid if and only if $\bB^{\varphi^{n}=\pi^{d}}$ has slope $\mu<1/k$.  
\end{thmx}

For   $d=1$, the Grassmannian  $\Gr_{k,C}(\bB^{\varphi^{n}=\pi})$ is isomorphic to the space $W_{n,n-k}$ defined in   \cite[Definition~5.1]{HM2022-1}. Theorem \ref{thm-A} therefore proves \cite[Conjecture~5.4]{HM2022-1}. 

Let us now explain the key ideas and constructions leading to Theorem \ref{thm-A}. Classically, the Pl\"{u}cker embedding provides a link between different Grassmannians; the same also holds in our setting, except now an additional parameter, the degree $d$, is  involved. Given a space   $\bB^{\varphi^{n}=\pi^{d}}$ of positive slope, put $N=\binom{n}{k}$ and $D=d\binom{n-1}{k-1}$. These are the rank and degree of $\bigwedge^{k}\cE_{n,d}$. Taking the determinant defines a Pl\"{u}cker morphism
\[
	\Pl: \Gr_{k}(\bB^{\varphi^{n}=\pi^{d}})\rightarrow \Gr_{1}(\bB^{\varphi^{N}=\pi^{D}}),
\] 
where a stably injective morphism $i: \cF\hookrightarrow \cE_{n,d}$ with $\cF\cong \OO^{k}$ at every geometric point is sent to $\wedge^{k}i: \det \cF\hookrightarrow \bigwedge^{k}\cE_{n,d}$. We then show:

\begin{thmx}[Theorem \ref{thm-iso-dec} and Proposition \ref{prop-Zariski-closed}]\label{thm-B}
	If  $\bB^{\varphi^{n}=\pi^{d}}$ has slope $<1/k$, the Pl\"{u}cker morphism $\Pl: \Gr_{k}(\bB^{\varphi^{n}=\pi^{d}})\rightarrow \Gr_{1}(\bB^{\varphi^{N}=\pi^{D}})$ is a Zariski closed immersion.
\end{thmx}

Since $D/N=kd/n$, Theorem \ref{thm-B} (after base change to $\Spd C$) reduces perfectoidness of the  $k$-Grassmannian of a space of slope $<1/k$ to the case of a projectivized space of slope $<1$. In this range, the absolute BC space is identified with the special fiber of the universal cover of a $p$-divisible group. By \cite[Proposition~3.1.3]{SW2013-1}, its punctured space, after quotienting by the torsion subgroup $\mu_{q-1}\subset \OO^{\times}_{K}$, admits a finite $\OO^{\times}_{K}$-stable cover in the analytic topology by affinoid perfectoid  $\cV_i=\Spa (V_i,V_i^{+})$, where 
\begin{align*}
	&V_i=\Fq\llparent u_i^{1/p^{\infty}}\rrparent \langle y^{1/p^{\infty}}_{ij}, 0\leq j\leq d-1 \text{ with }j\neq i\rangle \\
	&V_i^{+}=\Fq\llbracket u_i^{1/p^{\infty}}\rrbracket \langle y_{ij}^{1/p^{\infty}}, 0\leq j\leq d-1\text{ with }j\neq i\rangle
\end{align*}
For each unit subgroup $U_m=1+\pi^{m}\OO_K\subset \OO^{\times}_{K}$, we construct (Definition \ref{defn-orbit-norm}) certain norms for elements in $V_i$ with respect to $U_m$ by averaging its action and taking the limit. In particular, the resulting $U_m$-norm of $u_i$ is a pseudouniformizer of $V_i^{U_m}$, and the invariant pair $(V_i^{U_m}, V_i^{+U_m})$ is perfectoid. To identify $(V_i,V_i^{+})$ as a pro-\'etale $U_1$-torsor over $(V_i^{U_1}, V_i^{+U_1})$, the essential step  is to show the quotient at level  $U_{m+1}$ is a finite \'etale $\Fq$-torsor over the one at level $U_{m}$. For this, we construct (Lemma \ref{lem-theta}) an element $\theta_m\in V_i^{U_{m+1}}$ such that $g(\theta_m)-\theta_m\in (V_i^{U_{m+1}})^{\times}$ for every nontrivial $g\in \Fq$.  We also show (Proposition \ref{prop-Hansen}) such a condition in fact gives a general criterion for finite \'etale torsors. The constructions above yield the following result:

\begin{thmx}[Theorem \ref{thm-Gr1-perf}]\label{thm-C}
	If $\bB^{\varphi^{n}=\pi^{d}}$ has slope $<1$, the $v$-sheaf quotients   
	$$(\bB^{\varphi^{n}=\pi^{d}}\bs \{0\})/\underline{\OO^{\times}_{K}} \quad \text{and}\quad (\bB^{\varphi^{n}=\pi^{d}}_C\bs \{0\})/\underline{K^{\times}}$$ 
	are representable by perfectoid spaces.   
\end{thmx}

Together, Theorems \ref{thm-B} and \ref{thm-C} prove the positive direction of the main theorem. For the converse, suppose $d/n\geq 1/k$, one observes that the Grassmannian then contains a locus of injections which factor through the subbundle $\OO(1/k)\subset \cE_{n,d}$. This locus parametrizes exactly the degree 1 modifications of a vector bundle fibrewise isomorphic to $\OO^{k}$, and we show it is Zariski closed in the Grassmannian. For  $k=1$, this locus is simply $\Div^{1}_{C}$, which is not perfectoid (Proposition \ref{prop-div}). For $k>1$, consider the sublocus of modifications supported on the canonical divisor $i: S\rightarrow X_{S^{\tilt}}$. This is the $\GL_k(K)$-quotient of the infinite-level Lubin--Tate space $\cM^{\LT}_{\infty}$, hence isomorphic to $\P^{k-1}$. As $\P^{k-1}$ is not perfectoid, this completes the other direction of the main theorem.  

Theorem \ref{thm-A} has applications to Scholze's mod $p$ Jacquet--Langlands functor 
\[
	\cJ^{i}_{n}: \Rep^{\sm}_{\Fp}(\GL_n(K))\rightarrow \Rep^{\sm}_{\Fp}(D^{\times}_{1/n}). 
\]
Scholze's general bound \cite[Theorem~1.1]{Scholze2018-1} gives $\cJ^{i}_{n}(\rho)=0$ for $i>2(n-1)$ and every admissible $\Fp$-representation $\rho$ of $\GL_n(K)$. For $1\leq k\leq n-1$, denote by $P_k\subset\GL_n$ the standard maximal parabolic subgroup of block size $(n-k,k)$. Hansen--Mann's partial K\"{u}nneth formula \cite[Theorem~5.2]{HM2022-1} expresses Scholze's functor on parabolic inductions from $P_k(K)$ in terms of cohomology of  $\Gr_{n-k,C}(\bB^{\varphi^{n=\pi}})$, with coefficients in a complex of \'etale $\Fp$-sheaves involving $\cJ_{k}$. Since this Grassmannian has  dimension $n-k$ and the \'etale cohomology of a perfectoid space vanishes in degrees larger than its dimension, we obtain the following vanishing bounds for all admissible proper parabolic inductions for $\GL_n(K)$.

\begin{thmx}[Theorem \ref{thm-para-ind}]\label{thm-D}
	Let $0<k<n$. For any $\rho\in \Rep^{\adm}_{\Fp}(P_k(K))$, one has 
	\[
		\cJ^{i}_{n}(\Ind^{\GL_n(K)}_{P_k(K)}(\rho))=0 \quad \text{for $i>n+k-2$}. 
	\] 
\end{thmx}

Theorem \ref{thm-D} further  gives a vanishing theorem for  non-supersingular  admissible representations in the top  degree $2(n-1)$. Scholze's general bound implies that $\cJ^{2n-2}_{n}$ is right-exact.  
The classification theorem of Henniart--Vign\'eras \cite{HV2019-1} shows  that every non-supersingular irreducible $\rho \in \Rep^{\adm}_{\Fp}(\GL_n(K))$ is isomorphic to a parabolic induction of some Steinberg type representation. Upon further restricting to such $\rho$ which are infinite-dimensional, one gets a surjection from a  proper parabolic induction to $\rho$, from which we deduce the following vanishing result.

\begin{thmx}[Theorem \ref{thm-Herzig}]\label{thm-E}
	For any $n>1$ and any irreducible $\rho\in \Rep^{\adm}_{\Fp}(\GL_n(K))$ that is non-supersingular and infinite-dimensional, we have $\cJ^{2n-2}_n(\rho)=0$. 
\end{thmx}

Therefore, for such infinite-dimensional non-supersingular representations of $\GL_n(K)$, Scholze's bound can be improved to $2n-3$. Via explicit examples, one can see this new bound is in fact sharp, and the  assumption of infinite-dimension is also crucial.

\begin{rem*}[Related Work]
	The result of this text is obtained in June 2026. After the  first draft of this paper was prepared (September 2026), Ma independently proposed an inductive argument to reduce  the perfectoidness of $W_{n,k}$ of \cite{HM2022-1} to the perfectoidness of $W_{n,1}$ proved in \cite[Theorem~A]{JLH2021-1}. Our paper concerns more general absolute BC spaces beyond the dimension 1 case (see \S\ref{section-1.3}) and is self-contained. We also give a purely local proof of \cite[Theorem~A]{JLH2021-1}, see Remark \ref{rem-new-JLH}. 
\end{rem*}

\subsection*{Outline of the article}
For $S\in \Perf_{\Fq}$, let  $X_S=X_{S,K}$ be the relative Fargues--Fontaine curve. Fix an absolute BC space  $\bB^{\varphi^{n}=\pi^{d}}$ of slope $\mu>0$.  

\S1  reviews vector bundles on $X_S$, defines the \textit{Grassmannians} $\Gr_{k}(\bB^{\varphi^{n}=\pi^{d}})$ for $0<k<n$, and surveys known results on their representability. 

\S2 introduces the notion of \textit{orbit norm} for elements of an $\Fp$-algebra endowed with an  action of a filtered profinite group. After covering  the punctured BC space by affinoid perfectoid charts, it proves  their rings of    $\OO^{\times}_{K}$-invariants are perfectoid when $\mu<1$. 

\S3 realizes this geometrically by showing each chart is a pro-\'etale $\OO^{\times}_{K}$-torsor over the adic spectrum of its invariant ring. For $\bB^{\varphi^{n}=\pi^{d}}$ with $\mu<1$, it then proves its  projectivization is perfectoid after base change to $\Spd C$. 

\S4 constructs the \textit{Pl\"{u}cker morphisms} and identifies the $k$-Grassmannian of $\bB^{\varphi^{n}=\pi^{d}}$ with the decomposable locus in the projectivization of a BC space of slope $k\mu$. It deduces that the $k$-Grassmannian is perfectoid when $\mu<1/k$.

\S5 treats the rest cases of  $\mu\geq 1/k$. It proves that $\Div^{1}_{C}$ is not perfectoid, and  then establishes  the non-perfectoidness of the  $k$-Grassmannian by reduction to $\Div^{1}_{C}$ or $P^{k-1}$.

\S6 presents applications to representation theory. It recalls Scholze's functor, proves vanishing bounds of the functor on parabolic inductions and non-supersingular representations, and gives examples with explicit computations.

\subsection*{Notation and conventions}
We let $K/\Qp$ be a finite extension, with uniformizer $\pi$ and residue field $\Fq$. We use  $K_n$ to denote the  unramified degree $n$ extension of $K$.
For a perfect  $\Fp$-algebra $R$, $W(R)$ denotes its ring of Witt vectors and $[\cdot]: R\rightarrow W(R)$ denotes the Teichm\"{u}ller lift. The (relative) Fargues--Fontaine curve will be  defined with respect to $K$.  For $m\in \N^{\ast}$, $D_{1/m}$ is the division algebra over $K$ of invariant $1/m$. 

We fix  $F$  a complete algebraically closed nonarchimedean field over $\Fq$, $C/\Qp$ a complete algebraically closed nonarchimedean field with ring of integers $\OO_C$. The category of perfectoid spaces over $\Fq$ (resp. over $C$) is denoted $\Perf_{\Fq}$ (resp. $\Perf_{C}$). 

We reserve the letters $n,d$ for the rank and degree of an absolute Banach--Colmez space $\bB^{\varphi^{n}=\pi^{d}}$ and always assume $n\in \N^{\ast}, d\in \Z$. 
\subsection*{Acknowledgements}
I learned about  \cite[Conjecture~5.4]{HM2022-1} from Dave Hansen during the ``Geometric approaches to the local Langlands program" workshop at Brin Mathematics Research Center, and I am very thankful to  him for many helpful conversations and explanations on surrounding topics.  I am very grateful  to my advisor Jared Weinstein for sharing his insights, suggestions, and for his encouragement  throughout this project. I also want to thank Marie-France Vign\'eras for helpful correspondences and explanations.  I am also grateful to  Dave Hansen and Jared Weinstein for helpful comments on earlier drafts of this text. 
The author was partially supported by NSF grant DMS-2401472.

The author is a member of the AI-free caucus of Association of Human Mathematics.

\section{Absolute Banach--Colmez spaces}
In this section, we first recall vector bundles on the relative Fargues--Fontaine curve and their relationship to absolute BC spaces following \cite{Fargues2020-1} and \cite[Chapter~II]{FS2024-1}. We also review the moduli spaces $X_{n,k}$ and $W_{n,k}$ from  \cite[Appendix~A]{JLH2021-1} and \cite[Section~5]{HM2022-1}, which parametrize  homomorphisms between vector bundles, and compare them with the punctured absolute BC spaces. We then introduce  Grassmannians of absolute BC spaces, and survey some previous results of their representability. 
\subsection{Vector bundles on the Fargues--Fontaine curve and absolute BC spaces}
For $S=\Spa(R,R^{+})\in \Perf_{\Fq}$ with pseudo-uniformizer $\varpi\in R^{+}$, denote by 
\[
	Y_{S}=Y_{S,K}:= \Spa(W_{\OO_{K}}(R^{+}))\bs V(\pi[\varpi]).
\] 
The \textit{relative Fargues--Fontaine curve} is the quotient $X_S=X_{S,K}:=Y_S/\varphi^{\Z}$. Via gluing, such constructions also extend to general $S\in \Perf_{\Fq}$. For  an interval $I=[a,b]\subset (0,\infty)$ with rational ends, denote by 
\[
	Y_{S,I}:=\{\left|\pi\right|^{b}\leq \left|[\varpi]\right|\leq \left|\pi\right|^{a}\}\subset Y_{S},
\] 
which is an open subspace of $Y_{S}$. In this notation, $Y_S$ can be also written as $Y_{S,(0,\infty)}$.

We write $\bB$ for $v$-sheaf on $\Perf_{\Fq}$ given by $\bB(S)=\OO(Y_{S})$. It  has a Frobenius endomorphism $\varphi=\varphi_q$ coming from that of $Y_S$. When $S=\Spa(R,R^{+})$ is affinoid, we also denote $\OO(Y_S)$ by $B_S$ and $\OO(Y_{S,I})$ by $B_{S,I}$. The open subspace $Y_{S,I}$ is then given by $\Spa(B_{S,I}, B_{S,I}^{+})$, hence affinoid. 

Let $S\in \Perf_{\Fq}$ and assume $d\in \Z$, $n\in \N^{\ast}$. When $n,d$ are relatively prime, we write $\OO(d/n)$ for the stable vector bundle of slope $d/n$ on $X_{S}$. For a generic pair $(d,n)$, we write $\cE_{n,d}$ for the semistable vector bundle of degree $d$ and rank $n$ on $X_{S}$. According to the classification of vector bundles on $X_F=X_{\Spa(F,F^{+})}$ \cite{FF2018-1}, if $g=\gcd(d,n)$, then
\[
	\cE_{n,d}=\OO\left(\frac{d/g}{n/g}\right)^{\oplus g}.
\] 
Moreover, any vector bundle $\cE/X_F$ of slope $\mu$ is  always isomorphic to a direct sum of vector bundles of the form $\OO_{X_F}(d/n)$. When $\cE$ is semistable,    $H^{0}(X_F, \cE)\neq 0$ if and only if $\mu\geq 0$. 

For a vector bundle $\cE$ on $X_S$, the \textit{Banach--Colmez space} $\BC(\cE)$ associated to  $\cE$ is the locally spatial diamond over $\Perf_{S}$ defined by 
\[
	\BC(\cE)(T)=H^{0}(X_T, \cE_{T}), \quad T\in \Perf_{S}.
\] 
For  $d\geq 0$, the \textit{absolute Banach-Colmez space} $\bB^{\varphi^{n}=\pi^{d}}$ is the $v$-sheaf over $\Perf_{\Fq}$ defined by 
\[
	\bB^{\varphi^{n}=\pi^{d}}=\OO(Y_S)^{\varphi^{n}=\pi^{d}}=H^{0}(X_S, \cE_{n,d}), \quad S\in \Perf_{\Fq}.
\] 
Notice that $\bB^{\varphi^{n}=\pi^{d}}$ is not a diamond. For example, when $d=0$ and  $n=1$, one gets $\bB^{\varphi^{n}=\id}=\underline{K}$. However, pullback to  $S\in \Perf_{\Fq}$ gives 
\[
	\bB^{\varphi^{n}=\pi^{d}}_{S}:=\bB^{\varphi^{n}=\pi^{d}}\times_{\Spd \Fq} S=\BC(\cE_{n,d}),
\] 
hence  $\bB^{\varphi^{n}=\pi^{d}}_S$ is a locally spatial diamond over $S$. This also explains the notion ``absolute" (see \cite[\S7]{Fargues2020-1} for more details).

\begin{prop}[{\cite[Proposition~II.2.5 (iv)]{FS2024-1}}]\label{prop-B-presen}
	For $0<d\leq n\cdot [K:\Qp]$, there is an isomorphism of $v$-sheaves
	\[
	\bB^{\varphi^{n}=\pi^{d}}\cong \Spd \F_{q} \llbracket x_0^{1/p^{\infty}}, \ldots, x_{d-1}^{1/p^{\infty}}\rrbracket.
	\] 
\end{prop}
\begin{proof}
	The statement in \cite{FS2024-1} is originally over $\ol{\F}_{q}$; its  proof uses \cite[Proposition~3.1.3 (iii)]{SW2013-1}, which holds for any perfect ring of characteristic $p$. 
\end{proof}

Notice the presentation above only involves the degree $d$ of $\bB^{\varphi^{n}=\pi^{d}}$. As is later explained in Lemma \ref{lem-action}, the rank $n$ is related to the Frobenius action and also the $\pi$-action on the variables $x_i$. Also from this presentation, we see again $\bB^{\varphi^{n}=\pi^{d}}$ is  not a diamond as it contains the non-analytic point $\Spd \F_{q}$. 

Let $\bB^{\varphi^{n}=\pi^{d}}\bs \{0\}$ be the $v$-subsheaf of $\bB^{\varphi^{n}=\pi^{d}}$ containing sections  which are nonzero after pullback to any geometric point, which is called a \textit{punctured absolute Banach--Colmez space}. Specifically,  when $0<d\leq n\cdot [K:\Qp]$, we have
\begin{align}
	\bB^{\varphi^{n}=\pi^{d}} \bs \{0\} &\cong  \Spd \left( \F_{q}\llbracket x_0^{1/p^{\infty}},\ldots, x_{d-1}^{1/p^{\infty}}\rrbracket\right)\bs V(x_0,\ldots, x_{d-1})\nonumber\\
	&= \bigcup^{d-1}_{i=0}D(x_i)\nonumber
\end{align}
where each $D(x_i)$ is an open perfectoid ball over the perfectoid field $\F_{q}\llparent x_i^{1/p^{\infty}}\rrparent$. Therefore, we have:

\begin{prop}\label{prop-B*-presen}
	For $0<d\leq n\cdot [K:\Qp]$, the $v$-sheaf   $\bB^{\varphi^{n}=\pi^{d}}\bs \{0\}$ is representable by a perfectoid space.
\end{prop}

\subsection{Homomorphisms between vector bundles}
We now recall certain functors on $\Perf_{\Fq}$ parametrizing  homomorphisms between vector bundles  following \cite{Hansenetal2022-1} and \cite[Appendix~A]{JLH2021-1}. We then  compare them with the punctured absolute BC spaces, and define Grassmannians of  absolute BC spaces. 

 Let $\cE,\cF$ be vector bundles over $X_F$. The Hom-functor $\HHom(\cE, \cF)$ sends $S\in \Perf_{\Fq}$ to the set of $\OO_{X_S}$-module homomorphisms $\cE_{X_S}\rightarrow \cF_{X_S}$ and we have $\HHom(\cE, \cF)(S)\cong H^0(X_S, \cE^{\vee}\otimes \cF)$, where $\cE^{\vee}$ denotes the dual vector bundle of $\cE$.  Denote by  $\Surj(\cE, \cF)\subseteq \HHom(\cE, \cF)$ the subfunctor parametrizing   surjective homomorphisms. Also denote by  $\Inj(\cE, \cF)\subseteq \HHom(\cE, \cF)$ the subfunctor parametrizing   \textit{stably injective} homomorphisms, i.e. maps $\cE_{X_S}\rightarrow \cF_{X_S}$ which are  injective after pullback along any geometric point $\Spa (F,F^{+})\rightarrow S$. Both $\Inj(\cE, \cF)$ and $\Surj(\cE,\cF)$ are open subfunctors of $\HHom(\cE,\cF)$, and all three functors are locally spatial diamonds over $\Spd \Fq$ (see \cite[\S3]{Hansenetal2022-1} and \cite[Proposition~13.4]{Scholze-Diamond}).

Directly from their definitions, one notices that 
$$\Inj(\OO, \cE_{n,d})\cong \bB^{\varphi^{n}=\pi^{d}}\bs \{0\}.$$ 
Moreover, $\underline{K^{\times}}$ acts on both sides and the isomorphism between their $S$-points is $K^{\times}$-equivariant. Thus we have an isomorphism of quotient sheaves
\begin{equation}\label{eq-1.2.1}
	\Inj(\OO, \cE_{n,d})/\underline{K^{\times}}\cong (\bB^{\varphi^{n}=\pi^{d}}\bs \{0\})/\underline{K^{\times}},
\end{equation}
where the right hand side is a projectivized absolute BC space. The idea of studying such projectivized absolute BC  spaces first appeared in  \cite{Fargues2020-1}. Specifically, for $d\geq 0$, it was shown in loc. cit. that there is an isomorphism 
$$(\bB^{\varphi=\pi^{d}}\bs \{0\})/\ul{K}^{\times}\cong \Div^{d}$$ 
where $\Div^{d}$ is the sheaf sending $S\in \Perf_{\Fq}$ to the set of closed degree-$d$ Cartier divisors of $X_{S}$. However, such an identification with the moduli space of Cartier divisors does not exist for $\Inj(\OO, \cE_{n,d})/\ul{K}^{\times}$ when $n>1$, in which case one instead has
\[
	(\Inj(\OO, \cE_{n,d})\times_{\Spd \Fq}\Spd \F_{q^{n}})/\underline{K_n^{\times}}\cong \Div^{d}\times _{\Spd \Fq} \Spd \F_{q^{n}}.
\]

Nevertheless, the quotient $\Inj(\OO, \cE_{n,d})/\underline{K^{\times}}$ for  $n>1$ naturally leads to some other moduli spaces. 
As one can see from  (\ref{eq-1.2.1}),  $\Inj(\OO, \cE_{n,d})/\ul{K}^{\times}$ parametrizes lines  inside $\bB^{\varphi^{n}=\pi^{d}}$. Generalizing on this idea, we may  also parametrize higher dimensional subspaces of $\bB^{\varphi^{n}=\pi^{d}}$ by considering stably injective morphisms  from $\OO^{k}$ to  $\cE_{n,d}$.

\begin{defn}\label{defn-grass}
Let   $n,d\in \N^{\ast}$ and   $1\leq k\leq n-1$.  We define the \textit{$k$-Grassmannian} of the absolute Banach--Colmez space $\bB^{\varphi^{n}=\pi^{d}}$ as the quotient sheaf 
	\[
		\Gr_k(\bB^{\varphi^{n}=\pi^{d}}):=\Inj(\OO^{k}, \cE_{n,d})/\underline{\GL_k(K)}
	\]
	and we denote by $\Gr_{k,C}(\bB^{\varphi^{n}=\pi^{d}})$ its base change $\Gr_k(\bB^{\varphi^{n}=\pi^{d}})\times_{\Spd \Fq} \Spd C$. 
\end{defn}

\subsection{Representability of Grassmannians}\label{section-1.3}
After base change to $\Spd C$, the Grassmannian $\Gr_{k,C}(\bB^{\varphi^{n}=\pi^{d}})$ is a locally spatial diamond. The main goal of this text is to study when are such Grassmannians  representable by perfectoid spaces. 
The first examples of perfectoid Grassmannians come from the auxiliary spaces $X_{n,k}$ and $W_{n,k}$ studied in \cite{JLH2021-1} and \cite{HM2022-1}, and Definition \ref{defn-grass} is directly motivated by such spaces. 

We now recall the construction of $X_{n,k}$ and $W_{n,k}$. Let $\cM^{\LT}_{\infty}$ be the infinite-level Lubin--Tate space over $C$, which is perfectoid and admits a continuous $\GL_n(K)\times D^{\times}_{1/n}$-action \cite{SW2013-1}. The quotient $\cM^{\LT}_{\infty}/\GL_n(K)$ is identified with the projective space $\P^{n-1}$. For $1\leq k\leq n-1$, let $P_{k}(K)\subseteq \GL_n(K)$ be the maximal parabolic subgroup of the form  
\[
	P_{k}(K)=\twomat{\GL_{n-k}(K)}{\ast}{0}{\GL_{k}(K)}.
\] 
The  space $X_{n,k}$ is defined as the quotient $\cM^{\LT}_{\infty}/P_k(K)$. In the meantime, by passing to its associated diamond, $X_{n,k}$ also admits a moduli interpretation in terms of vector bundles on the relative Fargues--Fontaine curve. By \cite{SW2013-1}, the diamond $\cM^{\LT,\diamond}_{\infty}$ associated to $\cM^{\LT}_{\infty}$ can be  identified with the functor which assigns to $S\in \Perf_{C}$ the set of  modifications
\begin{equation}\label{eq-mod}
	0\rightarrow \OO^{n}\rightarrow \OO(1/n)\rightarrow i_{*}W\rightarrow 0
\end{equation}
over $X_{S^{\tilt}}$, where $i: S\rightarrow X_{S^{\tilt}}$ is the canonical closed immersion and $W$ is some rank one projective $\OO_S$-module. By \cite[Theorem~A.1.1]{JLH2021-1}, the quotient sheaf $\cM^{\LT,\diamond}_{\infty}/\ul{P_{k}(K)}$ then assigns to $S\in \Perf_C$ the set of diagrams 
\begin{equation}\label{eq-diagram}
	\OO(1/n)\surj \cE\hookleftarrow \cF
\end{equation}
of vector bundles over $X_{S^{\tilt}}$ with $\cE\cong \OO(1/k)$ and $\cF\cong \OO^{k}$ at all geometric points of $S^{\tilt}$ and $\cF\hookrightarrow \cE$ a modification of $\cF$ as in (\ref{eq-mod}). The space $W_{n,k}$ is then defined as the functor on $\Perf_{C}$ which parametrizes surjections of the form $\OO(1/n)\surj\cE$ over $X_{S^{\tilt}}$ with $\cE\cong \OO(1/k)$ at all geometric points of $S^{\tilt}$. In other words, $W_{n,k}$ remembers only the left half of (\ref{eq-diagram}). 

Before stating the representability results of $X_{n,k}$ and $W_{n,k}$, we first describe their relationship with the Grassmannians defined in Definition \ref{defn-grass}. Since $\OO(1/n)$ is stable, taking the kernel of a surjection $\OO(1/n)\surj \cE$ as in (\ref{eq-diagram}) gives a stably injective morphism $\cE'\hookrightarrow \OO(1/n)$ with $\cE'\cong \OO^{n-k}$ at every geometric point. Thus we have\footnote{In \cite[Appendix~A]{JLH2021-1}, Hansen defined the $\Inj$-functor and the $\Surj$-functor directly over $\Perf_{C}$, whereas we defined them over $\Fq$; hence we need to base change to $\Spd C$ to recover the space $W_{n,k}$.} 
\begin{align}
	W_{n,k} &= (\Surj(\OO(1/n), \OO(1/k))\times _{\Spd \Fq} \Spd C)/\underline{D^{\times}_{1/k}}\nonumber\\
	&\cong (\Inj(\OO^{n-k}, \OO(1/n))\times _{\Spd \Fq} \Spd C)/\underline{\GL_{n-k}(K)}\label{eq-W-inj}\\
	&= \Gr_{n-k, C}(\bB^{\varphi^{n}=\pi})\label{eq-W-B}. 
\end{align}
The space $X_{n,k}$ is not perfectoid when $k>1$ (\cite[Theorem~A.1.5]{JLH2021-1}); when $k=1$, one sees from (\ref{eq-diagram}) that $X_{n,1}=W_{n,1}$, hence we have 
\begin{equation}\label{eq-X-B}
	X_{n,1}\cong \Gr_{n-1,C}(\bB^{\varphi^{n}=\pi}). 
\end{equation}

\begin{thm}[{\cite[Theorem~4.2.1]{JLH2021-1}}]\label{thm-JLHmain}
The diamond $X_{n,1}$ is representable by a perfectoid space. 
\end{thm}

Johansson and Ludwig proved this result   by first exhibiting $\cM^{\LT}_{\infty}/P_1(K)$ as an open subspace of some infinite level Harris--Taylor Shimura variety. They then proved the latter is perfectoid via global methods.

\begin{thm}[{\cite[Corollary~A.1.4]{JLH2021-1}}]\label{thm-Hansen}
The diamond $W_{n,n-1}$ is representable by a perfectoid space. 
\end{thm}
Hansen's proof was originally written for the case of $n=2$, which  also applies to all $n>1$. We briefly recall  the argument here. 
\begin{proof}
	As in (\ref{eq-W-inj}), we have 
		$W_{n,n-1} \cong (\Inj(\OO, \OO(\frac{1}{n}))\times_{\Spd \Fq} \Spd C)/\underline{K^{\times}}$.
	Let $E/K$ be the unramified degree $n$ extension and let $E_{\infty}/E$ be  the Lubin--Tate extension of $E$, which has Galois group $\Gal(E_{\infty}/E)\cong \OO^{\times}_{E}$. Then for $S\in \Perf_{\F_{q^{n}}}$, we have 
	\begin{equation}\label{eq-Hansen-1}
		\Inj(\OO, \OO(1/n))(S)=H^{0}(X_{S,K}, \OO(1/n))^{\times}=H^{0}(X_{S,K_n}, \OO(1))^{\times},
	\end{equation}
	and \cite[Proposition~3.5.3]{Weinstein2017-1} further gives an $E^{\times}$-equivariant isomorphism
	\begin{equation}\label{eq-Hansen-2}
		H^{0}(X_{S,K_n},\OO(1))^{\times}\cong \Spd \hat{E}_{\infty}^{\tilt}.
	\end{equation}
	Therefore, combining (\ref{eq-Hansen-1}) and (\ref{eq-Hansen-2}), we get 
	\[
		\Inj(\OO, \OO(1/n))\times _{\Spd \Fq}\Spd \F_{q^{n}}\cong \Spd \hat{E}^{\tilt}_{\infty}.
	\] 
	Since $\hat{E}$ is perfectoid and $\Gal(E_{\infty}/E)\cong \OO^{\times}_{E}$, the subfield $L:=E_{\infty}^{\OO^{\times}_{K}}$ is again  deeply ramified over $K$, hence $\hat{L}$ is perfectoid. Thus we get 
	\begin{align*}
		W_{n,n-1} &\cong (\Spd \hat{E}^{\tilt}_{\infty}\times _{\Spd _{\F_{q^{n}}}} \Spd C)/\underline{K^{\times}}\\
		&\cong (\Spd \hat{E}^{\tilt}_{\infty}/\underline{\OO^{\times}_{K}}\times _{\Spd \F_{q^n}}\Spd C)/(\varphi_n\times \id)^{\Z}\\
		&\cong (\Spd \hat{L}^{\tilt}\times _{\Spd \F_{q^n}}\Spd C)/(\varphi_n\times \id)^{\Z},
	\end{align*}
	where $\varphi_n$ is the Frobenius of $K_n$.  
	Since both $\hat{L}^{\tilt}$ and $C$ are perfectoid and $\varphi_n\times \id$ acts properly discontinuously on the product,  $W_{n,n-1}$ is perfectoid. 
\end{proof}

These two theorems above give the perfectoidness of $W_{n,k}$ at the two ends when $k=1$ and $k=n-1$. Hansen and Mann then  conjectured the following: 

\begin{conj}[{\cite[Conjecture~5.4]{HM2022-1}}]\label{conj-HM}
	For every $1\leq k\leq n-1$, $W_{n,k}$ is perfectoid.  
\end{conj}

In terms of Grassmannians, Theorem \ref{thm-JLHmain} and Theorem \ref{thm-Hansen} show that $\Gr_{k,C}(\bB^{\varphi^{n}=\pi})$ is perfectoid when $k=1$ and $k=n-1$; the conjecture above further predicted the perfectoidness of all other Grassmannians (over $\Spd C$) of the slope $1/n$ absolute BC space $\bB^{\varphi^{n}=\pi}$.

\section{Orbit norm of formal coordinates}\label{section-2}
In this section, we  describe the  $K^{\times}$-action on an absolute Banach-Colmez space $\bB^{\varphi^{n}=\pi^{d}}$ and fix an $\OO^{\times}_{K}$-stable affinoid perfectoid open cover $\bigcup \cW_i$ of the punctured space $\bB^{\varphi^{n}=\pi^{d}}\bs \{0\}$. For each affinoid $\cW_i=\Spa (W_i,W_i^{+})$, we  construct $\OO^{\times}_{K}$-invariant elements in $W_i$, and use them to show each quotient  pair $(W_i^{\OO_K^{\times}}, W_i^{+ \OO^{\times}_{K}})$ is  perfectoid. Unless otherwise specified, we assume $0<d<n$ throughout this section. 
\subsection{Punctured BC spaces and $K^{\times}$-action}
Recall from Proposition \ref{prop-B*-presen} that when $0<d\leq n\cdot [K:\Qp]$, the punctured absolute BC space $\bB^{\varphi^{n}=\pi^{d}}\bs \{0\}$ is  perfectoid and admits a presentation 
\[
	\bB^{\varphi^{n}=\pi^{d}}\bs \{0\}\cong \Spd \Fq\llbracket x_0^{1/p^{\infty}}, \ldots, x_{d-1}^{1/p^{\infty}}\rrbracket \bs V(x_0,\ldots, x_{d-1}).
\] 
The formal variables  $x_i$ arise from the coordinates of a formal $\OO_K$-module and parametrize topologically nilpotent elements, thus we have $0\leq \left|x_i\right|<1$ for all $i$. Since the variables do not vanish simultaneously,  $\bB^{\varphi^{n}=\pi^{d}}\bs \{0\}$ admits an open cover
\begin{equation}\label{eq-cover}
	\bB^{\varphi^{n}= \pi^{d}}\bs \{0\}\cong \bigcup^{d-1}_{i=0}\cW_i, \quad \cW_i=\{\left|x_i\right|\neq 0, \left|x_j\right|\leq \left|x_i\right|\text{ for every $j$}\}.
\end{equation}
Let $y_{ij}=x_j/x_i$ for $j\neq i$,  we then have $0\leq \left|y_{ij}\right|\leq 1$ over $W_i$. Since $x_i$ is invertible over $W_i$, we can write $\cW_i$ more explicitly as $\Spa(W_i,W_i^{+})$ with 
\[
	W_i:=\F_{q}\llparent x_i^{1/p^{\infty}}\rrparent \langle y_{ij}^{1/p^{\infty}}, j\neq i\rangle \quad \text{and} \quad W_i^{+}:=\F_{q}\llbracket x_i^{1/p^{\infty}}\rrbracket \langle y_{ij}^{1/p^{\infty}}, j\neq i\rangle.
\] 
Therefore,  $\bigcup^{d-1}_{i=0}\cW_i$ gives an affinoid perfectoid open cover of $\bB^{\varphi^{n}=\pi^{d}}\bs \{0\}$.

The absolute BC space $\bB^{\varphi^{n}=\pi^{d}}$ is a $K$-vector space object  in the category of $v$-sheaves on $\Perf_{\Fq}$. Thus it  admits an action of $K^{\times}$. We now describe and fix the $K^{\times}$-action on the formal variables. 

\begin{lem}\label{lem-action}
The formal variables $x_0,...,x_{d-1}$ of $\bB^{\varphi^{n}=\pi^{d}}$ can be chosen such that 
\begin{renumerate}
	\item $\pi\in K^{\times}$ acts on $x_i$ by
		\[
	\pi(x_i)=\begin{cases}
	x_{i-1}^{q}  &  \text{if $1\leq i\leq d-1$}\\
	x_{d-1}^{q^{n-d+1}}  &  \text{if $i=0$}
	\end{cases}
\]
\item $c\in \Fq^{\times}$ with Teichm\"uller lift $[c]\in \OO^{\times}_{K}$ acts on $x_i$ by $$[c](x_i)=cx_i$$ for all $0\leq i\leq d-1$. 
\end{renumerate}
\end{lem}
\begin{proof}
	Let us first relabel the formal variables  of $\bB^{\varphi^{n}=\pi^{d}}$ as $\xi_0,\ldots,\xi_{d-1}$. For affinoid perfectoid $S=\Spa(R,R^{+})\in \Perf_{\Fq}$, an $S$-point of $\Spd \Fq \llbracket \xi_0^{1/p^{\infty}}, \ldots, \xi_{d-1}^{1/p^{\infty}}\rrbracket$ is a tuple $(r_0,\ldots, r_{d-1})\in (R^{\circ\circ})^{d}$ where $r_i$ is the value of the coordinate $\xi_i$. By \cite[Proposition~4.2.1]{FF2018-1}, there is a bijection  
	\begin{align*}
		\sL_{d,n}: (R^{\circ\circ})^{d}&\xrightarrow{\sim} \bB^{\varphi^{n}=\pi^{d}}(R)\\
		(r_0,...,r_{d-1})&\mapsto \sum^{d-1}_{i=0}\sum^{}_{m\in \Z}[r_i^{q^{-mn}}]\pi^{md+i}.
	\end{align*}
	Multiplying the image by $\pi$, we get $\pi\sL_{d,n}(r_0,\ldots,r_{d-1})=\sL_{d,n}(r^{q^{n}}_{d-1}, r_0,\ldots, r_{d-2})$. Matching this with the other side of  $\sL_{d,n}$, we see  $\pi$ acts on the coordinates $(\xi_i)$ by  
	\[
		\pi(\xi_0)=\xi^{q^{n}}_{d-1}, \quad \pi(\xi_i)=\xi_{i-1} \text{ for $1\leq i\leq d-1$}.
	\] 
	For each $i$, let $x_i=\xi^{q^{i}}_{i}$. This is an invertible change of coordinates because $R$ is  a perfect ring. The $\pi$-action on the variables $(x_i)$ is then given by 
\[
		\pi(x_0)=\xi^{q^{n}}_{d-1}=x^{q^{n-d+1}}_{d-1}, \quad \pi(x_i)=\xi^{q^{i}}_{i-1}=x^{q}_{i-1} \text{ for $1\leq i\leq d-1$}.
	\] 
For $c\in \Fq^{\times}$, since $c^{q}=c$, multiplication by its Teichm\"uller lift satisfies 
\[
	[c]\sL_{d,n}(\xi_0,\ldots,\xi_{d-1})=\sL_{d,n}(c \xi_0,\ldots,c\xi_{d-1}).
\] 
Hence we get  $[c](x_i)=(c\xi_{i})^{q^{i}}=cx_i$ as desired. 
\end{proof}

We \textbf{fix} these formal variables  $(x_i)$ for $\bB^{\varphi^{n}=\pi^{d}}$ together with the $K^{\times}$-action described in this lemma. Notice that our choice of $(x_i)$ is analogous to  the ones used in \cite[\S~2.3]{BHHMS2025-1}; their construction is the special case when $K/\Qp$ is unramified. 

The formal variables  $(x_i)$ can be identified with the coordinates of a formal $\OO_K$-module $\cG_{d,n}$ of dimension $d$ and height $n$ over $\Fq$ (see \cite[Exemple~4.3.2]{FF2018-1}). Let $\wt{\cG}_{d,n}=\invlim_{[\pi]}\cG_{d,n}$ be its universal cover. By \cite[Proposition~4.4.5]{FF2018-1}, we have an isomorphism 
	$\wt{\cG}_{d,n}\cong \bB^{\varphi^{n}=\pi^{d}}$.
Let $F=(F_0,\ldots,F_{d-1})$ denote the formal group law of $\cG_{d,n}$ in the chosen coordinates. For the $d$-dimensional formal variables $X,Y$ of $\cG_{d,n}$, since $F(X,0)=X$ and $F(0,Y)=Y$, the $i$-th component of $F(X,Y)$ is given by $F_i(X,Y)=X_i+Y_i+C_i(X,Y)$, where $C_i(X,Y)\in (X_0,\ldots,X_{d-1})(Y_0,\ldots,Y_{d-1})$ consists of higher order terms.

\begin{lem}\label{lem-higher-order}
	Let $I\subseteq \Fq\llbracket x_0,\ldots,x_{d-1}\rrbracket$ be the ideal generated by $x_0,\ldots,x_{d-1}$. 
	\begin{renumerate}
		\item For every $a\in \OO^{\times}_{K}$ with residue $\ol{a}\in \F^{\times}_{q}$, we have 
			\[
				a(x_i)\equiv \ol{a}x_i \mod I^{q}.
			\] 
		\item For $a=1+\pi b$ with  $b\in \OO_K$, we have 
			\[
				a(x_i)\equiv x_i+\ol{b} \pi (x_i) \mod I^{q+1}.
			\] 
		\item For $a=1+\pi^{m}b$ with $m\geq 1$ and $b\in \OO_{K}$, write $m=cd+f$ with $c\geq 0$ and $0\leq f<d$. We then have  
			\[
				a(x_i)\equiv x_i \mod I^{q^{cn+f}}.
			\] 
	\end{renumerate}
\end{lem}
\begin{proof}
	For $a\in \OO^{\times}_{K}$, write $a=[\ol{a}]+\pi a_1$ with $a_1\in \OO_K$. By Lemma \ref{lem-action}, we have  $[\ol{a}](x_i)=\ol{a}x_i$. As  every scalar endomorphism fixes the origin of $\wt{\cG}_{d,n}$, we have $a_1(I)\subseteq I$. Lemma  \ref{lem-action} also implies  $\pi(I)\subseteq I^{q}$, from which we get  $(\pi a_1)(I)\subseteq I^{q}$. Let $X=(x_0,\ldots,x_{d-1})$. By the formal group law, we have $a(x_i)=F_i([\ol{a}](X), (\pi a_1)(X))$.
	Since the second factor  lies in $I^{q}$, every higher order term lies in $I^{q+1}$, hence $a(x_i)\equiv \ol{a}x_i\mod I^{q}$.  This proves part (i), and part (ii) is the special case of (i) when $\ol{a}=1$. 

	For part (iii), we first compute the action of $\pi^{m}$. Notice that $\pi^{cd}(x_i)=x_i^{q^{cn}}$. In  the remaining $f$ iterations of $\pi$-action, the cyclic index crosses the special case  $x_0\mapsto x_{d-1}^{q^{n-d+1}}$ of part (i) of Lemma \ref{lem-action} exactly when $i<f$. Thus we have 
	\[
		\pi^{m}(x_i)=\begin{cases}
		x_{i-f}^{q^{cn+f}}  &  \text{if $i\geq f$}\\
		x_{i-f+d}^{q^{cn+f+n-d}}  &  \text{if $i<f$.}
		\end{cases}
	\] 
	Because $d<n$ by assumption, the exponent in both cases is  at least $q^{cn+f}$, hence we have $\pi^{m}I\subseteq I^{q^{cn+f}}$. By part (i), we have $b(x_i)\equiv \ol{b}x_i\mod I^{q}$, which implies $\pi^{m}b(x_i)\equiv \ol{b}\pi^{m}(x_i)\mod I^{q^{cn+f+1}}$ and  $\pi^{m}b(I)\subseteq I^{q^{cn+f}}$. Since $a(x_i)=F_i(X, \pi^{m}b(X))$, the higher order  terms belong to $I\cdot I^{q^{cn+f}}=I^{1+q^{cn+f}}$. As  $q^{cn+f+1}>1+q^{cn+f}$, we get $a(x_i)\equiv x_i+\ol{b}\pi^{m}(x_i) \mod I^{1+q^{cn+f}}$. Since $\pi^{m}(x_i)\in I^{q^{cn+f}}$, this proves part (iii). 
\end{proof}

\begin{prop}\label{prop-stable}
	Each affinoid perfectoid $\cW_i$ in the open cover (\ref{eq-cover}) of $\bB^{\varphi^{n}=\pi^{d}}\bs \{0\}$ is stable under the action of $\OO^{\times}_{K}$. 
\end{prop}
\begin{proof}
	Recall that each $\cW_i$ is defined by the condition  $\left|x_j\right|\leq \left|x_i\right|\neq 0$ for all $j$.  For  $a\in \OO^{\times}_{K}$, we show $\left|a(x_j)\right|\leq \left|a(x_i)\right|\neq 0$ for all $j$. By  part (i) of Lemma \ref{lem-higher-order}, we have $a(x_j)=\ol{a}x_j+\epsilon_j$ for some $\epsilon_j\in I^{q}$. Since  the variables $x_j$ are topologically nilpotent, we have $0\leq \left|x_j\right|<1$. Since $\epsilon_j\in I^{q}$, every monomial in $\epsilon_j$ has total degree at least $q$, hence we get $\left|\epsilon_j\right|\leq \left|x_i\right|^{q}<\left|x_i\right|$. Because $\left|\ol{a}\right|=1$, we also have   $\left|\epsilon_i\right|<\left|x_i\right|=\left|\ol{a}x_i\right|$. The non-archimedean triangle inequality then gives 
	\[
		\left|a(x_i)\right|=\left|\ol{a}x_i+\epsilon_i\right|=\left|x_i\right|, \quad 
	\left|a(x_j)\right|\leq \max \{\left|x_j\right|, \left|x_i\right|^{q}\}\leq \left|x_i\right|.
\] 
Therefore each $\cW_i$ is $\OO^{\times}_{K}$-stable. 
\end{proof}

\subsection{Tame quotients}
The group $\OO^{\times}_{K}$  has a decomposition 
\[
\OO^{\times}_{K}\cong \mu_{q-1}\times U_1,
\] 
where $\mu_{q-1}=[\Fq^{\times}]\subseteq \OO^{\times}_{K}$ is the \textit{tame subgroup} and $U_{1}=1+\pi \OO_{K}$ is the \textit{principal unit group}. For $m\geq 1$, let $U_m=1+\pi^{m}\OO_{K}$ be the \textit{$m$-th higher unit group} of $\OO^{\times}_{K}$. Such  $U_m$ form a  descending filtration
\[
\OO^{\times}_{K}\supseteq U_1\supseteq U_2\supseteq \cdots
\] 
with quotients $\OO^{\times}_{K}/U_1\cong \Fq^{\times}$ and $U_m/U_{m+1}\cong \Fq$. We now describe the tame quotient $\cW_i/\underline{\mu_{q-1}}$ for each chart $\cW_i$ of the punctured space.

\begin{prop}\label{prop-tame-quotient}
	Let $0\leq i\leq d-1$ and let  $u_i=x_i^{q-1}$. 
\begin{renumerate}
	\item For each $i$, let  $(V_i, V^{+}_{i})$ be the perfectoid Tate--Huber pair given by 
\[
	V_i=\F_{q}\llparent u_i^{1/p^{\infty}}\rrparent \langle y^{1/p^{\infty}}_{ij}, j\neq i\rangle, \quad V_i^{+}=\F_{q}\llbracket u_i^{1/p^{\infty}}\rrbracket \langle y^{1/p^{\infty}}_{ij}, j\neq i\rangle.
\] 
		Then for the affinoid perfectoid  $\cW_i=\Spa(W_i,W_i^{+})\subseteq\bB^{\varphi^{n}=\pi^{d}}\bs \{0\}$, we have $(W_i^{\mu_{q-1}},W_i^{+\mu_{q-1}})=(V_i,V_i^{+})$.

	\item Let $\cV_i=\Spa(V_i, V_i^{+})$. Then  we have 
		\[
			\cW_i/\underline{\mu_{q-1}}\cong \cV_i\quad\text{and}\quad (\bB^{\varphi^{n}=\pi^{d}}\bs \{0\})/\underline{\mu_{q-1}}\cong \bigcup^{d-1}_{i=0}\cV_i.
		\]
		In particular, the  quotient $(\bB^{\varphi^{n}=\pi^{d}}\bs \{0\})/\underline{\mu_{q-1}}$ is a perfectoid space.
	\end{renumerate}
\end{prop}
\begin{proof}
	Let  $c\in \Fq^{\times}$ and let $[c]\in \mu_{q-1}$ be its Teichm\"uller lift. From Lemma \ref{lem-action}, we get $[c](u_i)=u_i$ and $[c](y_{ij})=y_{ij}$.
	For $p$-power roots of $x_i$, Lemma \ref{lem-action} also implies $[c](x_i^{1/p^{r}})=c^{1/p^{r}}x_i^{1/p^{r}}$, where $c^{1/p^{r}}$ is uniquely determined in the perfect field $\Fq$. Hence we have 
	\[
		[c](u_i^{1/p^{r}})=u_i^{1/p^{r}},\quad [c](y^{1/p^{r}}_{ij})=y^{1/p^{r}}_{ij},
	\] 
	which shows that $V_i\subseteq (W_i)^{\mu_{q-1}}$ and $V_i^{+}\subseteq (W_i^{+})^{\mu_{q-1}}$. 
	As $q-1$ is prime to $p$, the map $c\mapsto c^{1/p^{r}}$ is an automorphism. This gives  $V_i=(W_i)^{\mu_{q-1}}$ and $V_i^{+}=(W_i^{+})^{\mu_{q-1}}$, which proves  part (i). 

	Since the tame subgroup $\mu_{q-1}$ is finite, \cite[Proposition~2.1.3]{CGJ2023-1} (cf. also \cite[Theorem~1.2]{Hansen2016-1}) together with part (i) imply $\cW_i/\underline{\mu_{q-1}}\cong \cV_i$. Because the open cover $\bigcup^{d-1}_{i=0}\cW_i$ of $\bB^{\varphi^{n}=\pi^{d}}\bs \{0\}$  is $\OO^{\times}_{K}$-stable by Proposition \ref{prop-stable}, such $\cV_i$ form an affinoid perfectoid open cover of $(\bB^{\varphi^{n}=\pi^{d}}\bs \{0\})/\underline{\mu_{q-1}}$, which proves part (ii).  
\end{proof}

From now on, we also  fix the affinoid perfectoid cover $$(\bB^{\varphi^{n}=\pi^{d}}\bs \{0\})/\underline{\mu_{q-1}}=\bigcup^{d-1}_{i=0}\cV_i, \quad \cV_i=\Spa(V_i,V_i^{+})$$ as described in the proposition above. 

On each chart $\cV_i$, we have an action of  $U_m$ for every $m\geq 1$. 
Recall from Lemma \ref{lem-higher-order} that  the estimates of the $\OO^{\times}_{K}$-action on the variables  $x_i$ are expressed $I$-adically for the  ideal $I=(x_0,\ldots,x_{d-1})$. On each chart $\cW_i$ of $\bB^{\varphi^{n}=\pi^{d}}\bs \{0\}$, since  $x_i$ is a topologically nilpotent unit, we can describe such estimates $x_i$-adically. Similarly on each chart $\cV_i$, such estimates can be stated  $u_i$-adically.

As one can see from part (iii) of Lemma \ref{lem-higher-order}, the $x_i$-adic order of the $U_m$-action is described in terms of $m$ modulo the cycle length $d$. For the $U_m$-action on $\bB^{\varphi^{n}=\pi^{d}}$, we fix the following constant. 

\begin{notn}\label{notn-Qm}
For every $m\geq 1$, write $m=cd+f$ with $0\leq f\leq d-1$. We denote by 
\[
C_m:=\frac{q^{cn+f}-1}{q-1}.
\] 
\end{notn}

For each higher unit subgroup $U_m\subset \OO^{\times}_{K}$, we now give an estimate of its action on the elements $u_i, y_{ij}\in V_i^{+}$. 
\begin{lem}\label{lem-U_m-est}
Let $m\geq 1$, $g\in U_m$, and $j\neq i$. We have 
\begin{renumerate}
	\item $g(u_i)/u_i -1, g(y_{ij})-y_{ij}\in u_i^{C_m}V_i^{+}$;
	\item $g(u_i^{\lambda }V_i^{+})=u_i^{\lambda }V_i^{+}$ for every $\lambda \in \Z[1/p]_{\geq 0}$. 
\end{renumerate}
\end{lem}
\begin{proof}
	For any $j$, denote by  $\delta _j:=(g(x_j)-x_j)/x_i$. Part (iii) of Lemma \ref{lem-higher-order} then gives $\delta _j\in x_i^{q^{cn+f}-1}W_i^{+}$. As $g$ commutes with the $\mu_{q-1}$-action, $\delta _j$ is $\mu_{q-1}$-invariant, hence $\delta _j\in u_i^{C_m}V_i^{+}$; in particular, we get  $(1+\delta _j)\in (V_i^{+})^{\times}$. Direct computation then gives
	\[
		\frac{g(u_i)}{u_i}-1=(1+\delta _i)^{q-1}-1, \quad g(y_{ij})-y_{ij}=\frac{\delta _j-y_{ij}\delta _i}{1+\delta _i},
	\] 
	from which part (i) follows. Since $g(u_i)/u_i\in (V_i^{+})^{\times}$, taking its $p$-power roots preserves the filtration induced by $u_i$. This proves part (ii). 
\end{proof}

\begin{lem}\label{lem-infty}
	For every $m\geq 1$, the sequence $(C_{m+r}/q^{r+1})_{r\geq 0}$ is increasing and tends to $\infty$ as $r\to\infty$. 
\end{lem}
\begin{proof}
From the  definition of $C_m$, we have $C_{m+1}\geq qC_m$ and $C_{m+d}\geq q^{n}C_m$. The first inequality proves that the sequence is increasing, and the second one further gives
\[
	\frac{C_{m+\ell d}}{q^{\ell d+1}}\geq q^{\ell (n-d)}\frac{C_m}{q}
\] 
	for any $\ell\geq 0$. Since $C_m>0$ and $d<n$, we get $C_{m+r}/q^{r+1}\to\infty$ as $r\to\infty$.
\end{proof}

\subsection{Orbit norm}
Let $G$ be a pro-$p$ group  equipped with a descending filtration 
\[
G=G_0\supseteq G_1\supseteq G_2\supseteq \cdots
\] 
formed by open normal subgroups $G_r$ such that $\bigcap^{}_{r\geq 0}G_r=\{1\}$. For each quotient $G/G_r$, a \textit{transversal} for $G/G_r$ is a subset of $G$ containing exactly one coset representative of each  coset in $G/G_r$. For each $r\geq 1$, let  $\Sigma _r$ be a transversal for $G/G_r$. We call $(\Sigma _r)_{r\geq 1}$ a \textit{compatible} system of transversals for (the filtration $G_{\bullet}$ of) $G$ if there exist transversals  $\Theta_r$ for $G_{r-1}/G_r$ such that 
\[
\Sigma _r=\Theta_1\cdots \Theta_r.
\] 

Let $A$ be a complete perfect topological $\Fp$-algebra with an action of $G$ as above.
\begin{defn}\label{defn-depth-norm}
	Let $r\geq 0$ and  let $\Sigma _r$ be a transversal for $G/G_r$. For an element $x\in A$, we call
\[
	N_{G}^{(r)}(x):=\left( \prod^{}_{\sigma \in \Sigma _r}\sigma(x)\right)^{1/[G:G_r]}
\] 
a \textit{depth-$r$ $G$-orbit norm} of $x$. 
\end{defn}

Notice that $N_G^{(r)}(x)$ depends on the choice of $\Sigma _r$. For fixed $\Sigma _r$, $N_G^{(r)}(x)$ is well-defined because $A$ is perfect.  Let  $(\Sigma _r)_{r \geq 0}$ be  a compatible system of transversals for $G$ with associated $(N_G^{(r)}(x))_{r\geq 0}$. If $N_G^{(r)}(x)$ converges in $A$ as $r\to\infty$ and the limit is independent of the chosen compatible transversals, we let 
\[
	N_G(x):=\lim_{r\to\infty}N^{(r)}_{G}(x). 
\] 

\begin{defn}\label{defn-orbit-norm}
	For $x\in A$, the element $N_G(x)$ defined above will be called the \textit{$G$-orbit norm}\footnote{The construction of $N^{(r)}_{G}(x)$ and $N_{G}(x)$ is motivated by an exercise given by Jared Weinstein at the Arizona Winter School in 2013; see \cite[Project~A]{Weinstein-AWS}.} of $x$. 
\end{defn}

In the rest of this article, we shall be concerned only with the case when $A$ is $V_i=\OO(\cV_i)$ or $W_i=\OO(\cW_i)$,  and $G$ is $U_m\subseteq \OO^{\times}_{K}$ for some $m\geq 1$ with  fixed filtration 
\[
G=U_m\supset G_1=U_{m+1}\supset \cdots\supset G_r=U_{m+r}\supset \cdots
\]

\begin{exmps}~{}
\begin{arenumerate}
	\item Suppose $p$ is odd and $K=\Qp$. Then we have $U_1=1+p\Zp$, which is topologically generated by $\gamma=1+p$. 
A compatible system of transversals for $U_1$ can be chosen as $\Sigma _r=\{\gamma ^{t}: 0\leq t<p^{r}\}$. 
\item For general finite extension $K/\Qp$, a compatible system of transversals for $U_1$ can be chosen as follows. For $r\geq 1$, we choose transversals $\Theta_r\subseteq U_1$ for the quotients $U_r/U_{r+1}$ as 
	$\Theta_r=\{1+\pi^{r}[c]: c\in \Fq\}$ 
	and then set $\Sigma _r=\Theta_1\cdots \Theta_r$. This gives $\left|\Sigma_r\right|=[U_1:U_{1+r}]=q^{r}$, and $\Sigma _r$ is explicitly given by 
	\[
		\Sigma _r=\{\prod^{r}_{t=1}(1+\pi^{t}[c_t]): c_1,\ldots, c_r \text{ vary independently over $\Fq$}\}.
	\] 
\end{arenumerate}
\end{exmps}

The following technical lemma gives an estimate of the difference between orbit norms at successive depths. 
\begin{lem}\label{lem-successive}
	Let $\lambda \in \Z[1/p]_{\geq 0}$, $f\in x_i^{\lambda }W_i^{+}$, and let $\{N^{(r)}_{U_m}(f)\}_{r\geq 0}$ be a compatible system of $U_m$-orbit norms of $f$. Suppose for every $r\geq 0$, there exists $e_r\in \Z[1/p]_{\geq 0}$ with $e_r\geq \lambda $ such that $g(f)-f\in x_i^{e_r}W_i^{+}$ for every $g\in U_{m+r}$, then we have $$N^{(r+1)}_{U_m}(f)-N^{(r)}_{U_m}(f)\in x_i^{\lambda +(e_r-\lambda )/q^{r+1}}W_i^{+}.$$
\end{lem}
\begin{proof}
	Let $(\Sigma _r)$ be the compatible transversals inducing such  orbit norms. 
	Since they are compatible, there exists transversal $\Theta_{r+1}$ for $U_{m+r}/U_{m+r+1}$ such that $\Sigma _{r+1}=\Sigma _r\Theta_{r+1}$. 
	Set $P_r=\prod^{}_{\sigma \in \Sigma _r}\sigma (f)$ so that $N_{U_m}^{(r)}(f)^{q^{r}}=P_r$. For $\sigma \in \Sigma _r$ and $\tau \in \Theta_{r+1}$, write 
	\[
		\sigma \tau (f)=\sigma (f)+\delta _{\sigma ,\tau }, \quad \delta _{\sigma ,\tau }=\sigma (\tau (f)-f).
	\] 
	Part (ii) of Lemma \ref{lem-U_m-est} implies 
	\[
		\sigma (f)\in x_i^{\lambda }W_i^{+}, \quad \delta _{\sigma ,\tau }\in x_i^{e_r}W_i^{+}.
	\] 
	From the definition of $N_{U_m}^{(r)}(f)$, we have 
	\[
		P_{r+1}=\prod^{}_{\sigma \in \Sigma _r}\prod^{}_{\tau \in \Theta_{r+1}}(\sigma (f)+\delta _{\sigma ,\tau }).
	\] 
	Within this product, terms containing no $\delta _{\sigma ,\tau }$-factor have product  $\prod^{}_{\sigma \in \Sigma _r}\sigma (f)^{q}=P_r^{q}$. In the product of the rest terms, there are $k\geq 1$ factors of order at least $e_r$ and $q^{r+1}-k$ factors of order at least $\lambda $, giving total order at least 
	\[
		ke_r+(q^{r+1}-k)\lambda  =q^{r+1}\lambda  +k(e_r-\lambda )\geq q^{r+1}\lambda  +(e_r-\lambda  ). 
	\] 
	Thus we have 
	\[
		P_{r+1}-P_r^{q}\in x_i^{q^{r+1}\lambda  +(e_r-\lambda  )}W_i^{+}.
	\] 
	Since $W_i^{+}$ is perfect of characteristic $p$, we have 
	\[
		(N^{(r+1)}_{U_m}(f)-N^{(r)}_{U_m}(f))^{q^{r+1}}=N^{(r+1)}_{U_m}(f)^{q^{r+1}}-N^{(r)}_{U_m}(f)^{q^{r+1}}=P_{r+1}-P^{q}_{r}.
	\] 
	Taking the unique $q^{r+1}$-th root then proves the claim. 
\end{proof}

We now give a criterion for the existence of the orbit norm. 
\begin{lem}\label{lem-exist}
	Under the hypothesis of Lemma \ref{lem-successive}, if we further have $$\frac{e_r-\lambda  }{q^{r+1}}\to\infty,$$ then the orbit norm $N_{U_m}(f)$ exists and belongs to $W_i^{+U_m}$. Moreover, if $\beta\in \Z[1/p]_{\geq 0}$ satisfies $\lambda +(e_r-\lambda )/q^{r+1}\geq \beta$ for every $r\geq 0$, then $N_{U_m}(f)-f\in x_i^{\beta}W_i^{+}$. 
\end{lem}
\begin{proof}
	By Lemma \ref{lem-successive}, we have 
	\begin{equation}\label{eq-exist-1}
		N^{(r+1)}_{U_m}(f)-N^{(r)}_{U_m}(f)\in x_i^{\lambda +(e_r-\lambda )/q^{r+1}}W_i^{+}.
	\end{equation}
	As the exponent of $x_i$ goes to infinity by assumption, the sequence $(N^{(r)}_{U_m}(f))_{r}$ is Cauchy and converges in $W_i^{+}$. We are left to show the limit is independent of  the chosen transversals. Suppose $(\Sigma'_r)_{r}$ is another compatible system of transversals for $(U_m/U_{m+r})_r$ and let $N^{(r)}_{U_m,\Sigma _r'}(f)$ be the corresponding depth-r norms. For each $\Sigma _r'$, matching its representatives with those of $\Sigma _r$ gives 
	\[
	\Sigma _r'=\{\sigma h_{\sigma }: \sigma \in \Sigma _r\}
	\] 
	where each $h_{\sigma }$ belongs to $U_{m+r}$. By assumption of Lemma \ref{lem-successive} and part (ii) of Lemma \ref{lem-U_m-est}, we have $\sigma h_{\sigma }(f)-\sigma (f)\in x_i^{e_r}W_i^{+}$. Taking the product over the transversals $\Sigma _r$ and $\Sigma _r'$ respectively and taking the $q^{r}$-th root, we get 
		\[
		N^{(r)}_{U_m, \Sigma _r'}(f)-N^{(r)}_{U_m, \Sigma _r}(f)\in x_i^{\lambda +(e_r-\lambda )/q^{r}}W_i^{+}.
	\]
	As the exponent of $x_i$ goes to infinity, we see that $N^{(r)}_{U_m, \Sigma _r'}(f)$ converges to the same limit as $N^{(r)}_{U_m, \Sigma _r}(f)$. Thus the orbit norm $N_{U_m}(f)$ exists. To see that $N_{U_m}(f)$ is $U_m$-invariant, notice that any $g\in U_m$ gives another system of transversals $(g\Sigma _r)_{r}$. As the orbit norm is independent of the transversals, we get $g(N_{U_m}(f))=N_{U_m}(f)$. 

	Finally, if the exponent in (\ref{eq-exist-1}) is bounded below by $\beta$, summing up successive differences gives $N^{(r)}_{U_m}(f)-f\in x_i^{\beta}W_i^{+}$. Passing to the limit  gives  $N_{U_m}(f)-f\in x_i^{\beta}W_i^{+}$. 
\end{proof}

\begin{rem}
	Notice that the  $U_m$-orbit norm of $f$ being $U_m$-invariant is almost  formal and  ``built-in" from its construction. In general, given $G$ a filtered pro-$p$ group acting continuously on a complete perfect topological $\Fp$-algebra $A$, if $N_G(x)$ exists for some $x\in A$, it is natural to expect $N_G(x)$ is $G$-invariant.   
\end{rem}

For each chart $\cV_i=\Spa (V_i,V_i^{+})$ of the tame quotient $(\bB^{\varphi^{n}=\pi^{d}}\bs \{0\})/\underline{\mu_{q-1}}$, we now show that the elements $u_i, y_{ij}\in V_i^{+}=\Fq\llbracket u_i^{1/p^{\infty}}\rrbracket\langle y^{1/p^{\infty}}_{ij}, j\neq i\rangle$ admit $U_m$-orbit norm for every higher unit group $U_m\subset \OO^{\times}_{K}$. 
\begin{prop}\label{prop-Cauchy}
	Assume $0<d<n$. For every chart $\cV_i$ of $(\bB^{\varphi^{n}=\pi^{d}}\bs \{0\})/\underline{\mu_{q-1}}$ and every $m\geq 1$, the orbit norms $N_{U_m}(u_i)$ and $N_{U_m}(y_{ij})$ exist in $V_i^{+}$ and are $U_m$-invariant. Moreover, we have 
		\[
			\frac{N_{U_m}(u_i)}{u_i}-1\in u_i^{C_m/q}V_i^{+}, \quad N_{U_m}(y_{ij})-y_{ij}\in u_i^{C_m/q}V_i^{+}.
		\]
\end{prop}
\begin{proof}
	We first work in $W_i^{+}$. Let $r\geq 0$ and $g\in U_{m+r}$. Since $u_i=x_i^{q-1}$, we have $u_i\in x_i^{q-1}W_i^{+}$ and $y_{ij}\in W_i^{+}$, hence  part (i) of Lemma \ref{lem-U_m-est} gives 
	\begin{align*}
		 &g(u_i)-u_i\in x_i^{(q-1)(C_{m+r}+1)}W_i^{+}\\
		&g(y_{ij})-y_{ij}\in x_i^{(q-1)C_{m+r}}W_i^{+}.
	\end{align*}
	For both $u_i$ and $y_{ij}$, the ratio $(e_r-\lambda )/q^{r+1}$ from Lemma \ref{lem-exist} is given by
	\[
		\frac{(q-1)(C_{m+r}+1)-(q-1)}{q^{r+1}}=\frac{(q-1)C_{m+r}}{q^{r+1}},
	\] 
	which goes to infinity as $r\to\infty$ by Lemma \ref{lem-infty} and is bounded below by $(q-1)C_m/q$. Applying Lemma \ref{lem-exist} then shows the existence of $N_{U_m}(u_i)$ and $N_{U_m}(y_{ij})$, proves they are $U_m$-invariant, and also gives
	\[
		N_{U_m}(u_i)-u_i\in x_i^{(q-1)(1+C_m/q)}W_i^{+}, \quad N_{U_m}(y_{ij})-y_{ij}\in x_i^{(q-1)C_m/q}W_i^{+}.
	\] 
	Now passing to $V_i^{+}$, since $V_i^{+}$ is closed in $W_i^{+}$ and all the intermediate depth-$r$ norms belong to  $V_i^{+}$ by construction, we get $N_{U_m}(u_i), N_{U_m}(y_{ij})\in V_i^{+U_m}$.  
\end{proof}

\begin{cor}\label{cor-norm-prop}
	For every $m\geq 1$ and every $\ell\geq 0$, we have 
\begin{renumerate}
	\item $N_{U_m}(u_i)$ is a topologically nilpotent unit in $V_i$;
	\item $N_{U_m}(u_i)^{1/p^{\ell}}, N_{U_m}(y_{ij})^{1/p^{\ell}}$ belongs to $V_i^{+U_m}$;
	\item $N_{U_m}(u_i)^{1/p^{\ell}}$ is a topologically nilpotent unit in $V_i$. 
\end{renumerate}
\end{cor}
\begin{proof}
	Since $C_m>0$, Proposition \ref{prop-Cauchy} gives
	\[
		\frac{N_{U_m}(u_i)}{u_i}\in 1+u_i^{C_m/q}V_i^{+}\subset 1+V_i^{\circ\circ}\subset (V_i^{+})^{\times}.
	\] 
	Since $u_i$ is a topologically nilpotent unit, we get (i) and (iii).  Part (ii) follows directly from the perfectness of $V_i^{+}$. 
\end{proof}

\subsection{Perfectoid invariant rings}
We conclude this section by showing the $\OO^{\times}_{K}$-invariant subring of each $W_i$ is perfectoid.

\begin{prop}\label{prop-ring-perf}
	For every $m\geq 1$, $(V_i^{U_m}, V_i^{+ U_m})$ is a perfectoid Tate--Huber pair. In particular, $(W_i^{\OO^{\times}_{K}}, W_i^{+\OO^{\times}_{K}})$  is  a perfectoid Tate--Huber pair. 
\end{prop}
\begin{proof}
	By \cite[Proposition~3.5]{Scholze-Diamond}, to show $V_i^{U_m}$ is perfectoid, it is enough to show it is a perfect complete Tate ring. 
	By Proposition \ref{prop-Cauchy} and Corollary \ref{cor-norm-prop},  $N_{U_m}(u_i)$ is a topologically nilpotent unit in $V_i^{+U_m}$ with $N_{U_m}(u_i)=u_i\epsilon_m$ for some $\epsilon_m\in (V_i^{+})^{\times}$. 
	Thus for every $s\geq 0$, we get $N_{U_m}(u_i)^{s}V_i^{+}=u_i^{s}V_i^{+}$, where $\{u_i^{s}V_i^{+}\}_{s\geq 0}$ form a neighbourhood basis of 0 in $V_i$. 
	Since we also have $V_i^{U_m}\cap N_{U_m}(u_i)^{s}V_i^{+}=N_{U_m}(u_i)^{s}V_i^{+U_m}$, 
	$\{N_{U_m}(u_i)^{s}V_i^{+U_m}\}_{s\geq 0}$ form a neighbourhood basis of 0 in $V_i^{U_m}$. 
	Hence $V_i^{+U_m}$ is an open bounded ring of definition, and $V_i^{U_m}$ is a Tate ring with pseudouniformizer $N_{U_m}(u_i)$. 
	As $V_i^{U_m}=\bigcap^{}_{g\in U_m}\ker(g-\id)$, $V_i^{U_m}$ is closed in $V_i$, hence complete. 
	To show  it  is perfect, let $v\in V_i^{U_m}$ and let  $w\in V_i$ be its unique $p$-th root. 
	For every $g\in U_m$, we have $g(w)^{p}=g(v)=v=w^{p}$. Since $V_i$ is perfect, this gives $g(w)=w$, hence $w\in V_i^{U_m}$. 
	Thus Frobenius on $V_i^{U_m}$ is bijective, hence $V_i^{U_m}$ is perfect. This shows $V_i^{U_m}$ is perfectoid. 
	As $V_i^{+}$ is integrally closed in $V_i$ and $V_i^{+U_m}=V_i^{U_m}\cap V_i^{+}$,  $V_i^{+U_m}$ is also integrally closed in $V_i^{U_m}$, whence the assertion. 
\end{proof}

\begin{rem}
	Notice that although we constructed $U_m$-invariant elements $N_{U_m}(u_i)$ and $N_{U_m}(y_{ij})$ in $V_i$, the proof of perfectoidness of $(V_i^{U_m}, V_i^{+U_m})$ above did not really make use of $N_{U_m}(y_{ij})$. It would be ideal if we actually have 
	\begin{align*}
		&V_i^{U_m} = \Fq\llparent N_{U_m}(u_i)^{1/p^{\infty}}\rrparent \langle N_{U_m}(y_{ij})^{1/p^{\infty}}, j\neq i\rangle\\
		&V_i^{+U_m} = \Fq\llbracket N_{U_m}(u_i)^{1/p^{\infty}}\rrbracket \langle N_{U_m}(y_{ij})^{1/p^{\infty}}, j\neq i\rangle,
	\end{align*}
	but we do not know whether this is true at the moment. 
\end{rem}

\section{Projectivized Banach--Colmez spaces}\label{section-3}
In this section, we show the projectivization of an absolute Banach--Colmez space $\bB^{\varphi^{n}=\pi^{d}}$ of slope $<1$ is perfectoid. We assume $0<d<n$ throughout this section.  

\begin{notn}\label{notn-uniform}
	To ease the notation, we may sometimes suppress the index $0\leq i\leq d-1$. Thus for the chart  $\cW_i=\Spa (W_i,W_i^{+})$ of $\bB^{\varphi^{n}=\pi^{d}}\bs \{0\}$ and the chart $\cV_i=\Spa (V_i, V_i^{+})$ of  the tame quotient, we simply write them respectively as 
	\begin{align*}
		&\cW=\Spa(W, W^{+})=\Spa \Fq \llparent x^{1/p^{\infty}}\rrparent \langle y_j^{1/p^{\infty}}\rangle\\
		&\cV=\Spa (V, V^{+})=\Spa \Fq \llparent u^{1/p^{\infty}}\rrparent \langle y_j^{1/p^{\infty}}\rangle.
	\end{align*}
When it is clear, we may also just write $y$ for $y_j$. Such a convention is harmless, since all the constructions and computations below apply verbatim for every $i$. We shall restore the indices only when comparing different charts or their intersections. 
\end{notn}

Motivated by \cite[Theorem~1.2]{Hansen2016-1}, we give the following criterion for a Tate--Huber pair to be a finite \'etale torsor over its quotient by a finite group.
\begin{prop}\label{prop-Hansen}
	Let $(A,A^{+})$ be a complete Tate-Huber pair  with a continuous action of a finite group $G$ and put  $(B,B^{+})=(A^{G}, A^{+G})$. Suppose there exists $a\in A$ such that $g(a)-a\in A^{\times}$ for every nontrivial $g\in G$, then $(B,B^{+})$ is a complete Tate-Huber pair with $A=B[a]$ and $A/B$ is a finite \'etale $G$-torsor. 
\end{prop}
The idea is that the elements $g(a)$, for $g\in G$, acts as  $\left|G\right|$  distinct points if $g(a)-a\in A^{\times}$. Thus the condition separates the extension $A/B$. 
\begin{proof}
	The pair $(B,B^{+})$ being a complete Tate--Huber pair is standard, so  we focus on the rest of the claim.  
	For such an element $a\in A$, denote $g(a)$ by $a_g$. Then since $g(a)-a\in A^{\times}$ for every nontrivial $g\in G$, we have  $a_g-a_h\in A^{\times}$ for $g\neq h$. For each $g\in G$, consider 
	\[
		L_g(T):=\prod^{}_{\substack{h\in G \\ h\neq g}}\frac{T-a_h}{a_g-a_h}\in A[T].
	\] 
	We have  $L_g(a_h)=1$ if $g=h$ and $L_g(a_h)=0$ if $g\neq h$. The group action permutes these polynomials by $h(L_g(T))=L_{hg}(T)$ for $h,g\in G$. For $x\in A$, set 
	\[
		H_x(T):=\sum^{}_{g\in G}g(x)L_g(T).
	\] 
	Since $G$ permutes its summands, we get $H_x(T)\in B[T]$. As $H_x(a)=x$, this proves $A=B[a]$. 

	Now consider 
	\[
		P(T):=\prod^{}_{g\in G}(T-a_g) \in B[T].
	\] 
	Since $P(a)=0$ and $A=B[a]$, evaluation at $a$ gives a surjection from $B[T]/(P)$ to $A$. We show it is also injective. Suppose we have $F(T)\in B[T]$ of degree less than $\left|G\right|$ and satisfies $F(a)=0$. Since its coefficients are $G$-invariant, we also get $F(a_g)=0$ for any $g\in G$. Since the differences $a_g-a_h$ belong to $A^{\times}$ for $g\neq h$ by assumption, successive division by $T-a_g$ shows $P(T)$ divides $F(T)$. As $P$ is monic and $\deg(F)<\left|G\right|$, this forces $F(T)=0$. Thus  we get $A\cong B[T]/(P(T))$. Since we also have  
	\[
		P'(a)=\prod^{}_{g\neq 1}(a-a_g)\in A^{\times},
	\] 
	the extension $A/B$ is a finite \'etale $G$-torsor. 
\end{proof}

We now construct such an extension-separating element for each $V^{U_{m+1}}$ over $V^{U_{m}}$. 
\begin{lem}\label{lem-theta}
For every $m\geq 1$, there exists $\theta_m\in V^{U_{m+1}}$ such that 
\[
	g(\theta_m)-\theta_m\in (V^{U_{m+1}})^{\times}
\] 
for every nontrivial $g\in U_{m}/U_{m+1}$. 
\end{lem}
\begin{proof}
Suppose $V=V_i$ and write $m=cd+f$ with $0\leq f<d$. Consider the $d$-dimensional formal variable  $X=(x_0,\ldots, x_{d-1})$. We set  $Z=\pi^{-f}X$ and let  $\xi_{i}$ be the $i$-th coordinate of $Z$. More explicitly, using the description of $\pi$-action given in Lemma \ref{lem-action}, we have 
	\begin{equation}\label{eq-theta-1}
	\xi_{i}=\begin{cases}
	x_{i+f}^{1/q^{f}}  &  \text{if $i+f<d$}\\
	x^{1/q^{n-d+f}}_{i+f-d}  &  \text{if $i+f\geq d$}
	\end{cases}
	\end{equation}
We then define 
	\[
		\theta_m:=\frac{N_{U_{m+1}}(\xi_{i})}{N_{U_m}(x_i)^{q^{cn}}}.
	\] 
	By Proposition \ref{prop-Cauchy}, $\theta_m$ is $U_{m+1}$-invariant.  Since  $\mu_{q-1}$ acts on the two orbit norms by the same scalar, we also get $\theta_m\in V_i^{U_{m+1}}$. 

	Consider an element $g=1+\pi^{m}b\in U_m$ with $b\in \OO^{\times}_{K}$. We first compute the action of $g$ on $\xi_{i}$. By (\ref{eq-theta-1}), we can choose $0<\lambda \leq 1$ in $\Z[1/p]$ such that $\xi_j\in x_i^{\lambda }W_i^{+}$ for every $j$. Then using the formal $\OO_K$-module law, we have 
	\[
		g(Z)=F(Z, b\pi^{cd}(X))=F(Z, b(X^{cn})).
	\] 
	Since $b(x_j)-\ol{b}(x_j)\in x_i^{q}W_i^{+}$ by Lemma \ref{lem-higher-order}, we get 
	\begin{equation}\label{eq-theta-2}
		g(\xi_i)-\xi_i-\ol{b}(x_i^{q^{cn}})\in x_i^{q^{cn}+\lambda }W_i^{+}. 
	\end{equation}

	We now compute the difference between $\xi_i$ and its $U_{m+1}$-orbit norm. For $r\geq 0$, consider $h=1+\pi^{m+1+r}b'\in U_{m+1+r}$, then we have 
	\[
		h(Z)=F(Z,Y), \quad Y=b'\pi^{cd+r+1}(X).
	\] 
	By Lemma \ref{lem-U_m-est}, every component of $Y$ belongs to $x_i^{1+(q-1)C_{cd+r+1}}W_i^{+}$. As every mixed term of $F(Z,Y)$ contains a $Y$-component, we then get
	\[
		h(\xi_i)-\xi_i\in x_i^{1+(q-1)C_{cd+r+1}}W_i^{+}.
	\] 
	By Lemma \ref{lem-infty}, as $r\to\infty$, we have $$\frac{1+(q-1)C_{cd+r+1}-\lambda }{q^{r+1}}\to\infty;$$
	direct comparison also shows  $1+(q-1)C_{cd+r+1}\geq q^{cn+r+1}$. Thus applying Lemma \ref{lem-exist}, we get
	\begin{equation}\label{eq-theta-3}
		N_{U_{m+1}}(\xi_i)-\xi_i\in x_i^{q^{cn}+\lambda (1-1/q)}W_i^{+}.
	\end{equation}

	Finally, since $N_{U_m}(x_i)$ is $U_m$-invariant, (\ref{eq-theta-2}) and (\ref{eq-theta-3}) together give 
	\[
		g(\theta_m)-\theta_m=\left( \frac{x_i}{N_{U_m}(x_i)}\right)^{q^{cn}}(\ol{b}+\epsilon)
	\] 
	for some $\epsilon \in W_i^{\circ\circ}$. Thus $g(\theta_m)-\theta_m$ is a unit in $W_i$; it also lies in $V_i^{U_{m+1}}$ because $\theta_m$ lies in $V_i^{U_{m+1}}$. Hence $g(\theta_m)-\theta_m$ is a unit in $V_i^{U_{m+1}}$.  
\end{proof}

For $m\geq 1$, put  $(R_m, R^{+}_{m})=(V^{U_m}, V^{+U_m})$ and let $\cR_m=\Spa(R_m, R_m^{+})$. Such $\cR_m$ form a tower of affinoid perfectoid spaces. 
\begin{prop}\label{prop-torsor}
The natural maps $R_{m}\rightarrow R_{m+1}$ induce  a tower of finite \'etale torsors
\[
\cdots\rightarrow \cR_{m+1}\rightarrow \cR_{m}\rightarrow \cdots \rightarrow \cR_1
\] 
where $\cR_n\rightarrow \cR_m$ is a finite \'etale $U_m/U_n$-torsor for every $1\leq m<n$. 
\end{prop}
\begin{proof}
	Using the elements  $\theta_m$ constructed in Lemma \ref{lem-theta}, the criterion from Proposition \ref{prop-Hansen} implies that $R_{m+1}$ is a finite \'etale $U_{m}/U_{m+1}$-torsor over $R_m$. Since $R_{m+1}^{+}$ is the integral closure of $R^{+}_{m}$ in $R_{m+1}$, geometrically $\cR_{m+1}\rightarrow \cR_{m}$ is also a finite \'etale $U_m/U_{m+1}$-torsor (cf. \cite[Definition~6.2]{Scholze-Diamond})  with \v{C}ech groupoid given by 
	\[
		(U_m/U_{m+1})\times \cR_{m+1}\rightrightarrows \cR_{m+1}.
	\] 
	Since $U_m$ are normal subgroups of $U_1$ and all the actions are induced by the same $U_1$-action, induction gives the general case for $m<n$. 
\end{proof}

\begin{lem}\label{lem-density}
$\bigcup^{}_{m\geq 1}R_m^{+}$ is dense in $V^{+}$ and $\bigcup^{}_{m\geq 1}R_m$ is dense in $V$. 
\end{lem}
\begin{proof}
	By Proposition \ref{prop-Cauchy}, we have 
	\[
		\frac{N_{U_m}(u)}{u}-1, N_{U_m}(y)-y \in u^{C_m/q}V^{+},
	\] 
	where, recall from Notation \ref{notn-Qm}, for $m=cd+f$ with $0\leq f<d$ we set $C_m=(q^{cn+f}-1)/(q-1)$. Since $C_m\to\infty$ as $m\to\infty$, the orbit norms $N_{U_m}(u)$ and $N_{U_m}(y)$ converge to $u$ and $y$ respectively. Similarly, for every $\ell\geq 0$, the $p^{\ell}$-th roots of $N_{U_m}(u)$ and $N_{U_m}(y)$ also converge to the $p^{\ell}$-th roots of $u$ and $y$ respectively; such $p$-power roots also belong to $R^{+}_{m}$ by part (ii) of Corollary \ref{cor-norm-prop}. Therefore, we get
	\begin{equation}\label{eq-density-1}
		\ol{\bigcup^{}_{m\geq 1}R^{+}_{m}}\supseteq \Fq[u^{1/p^{\infty}}, y^{1/p^{\infty}}],
	\end{equation}
	and the latter is dense in $V^{+}$. This proves the first claim. Now part (i) of Corollary \ref{cor-norm-prop}	 gives $N_{U_1}(u)\in R^{\times}_{1}$, hence we get  $N_{U_1}(u)/u\in (V^{+})^{\times}$ and  $V=V^{+}[1/N_{U_1}(u)]$. As $N_{U_1}(u)^{-1}\in R_1\subset R_m$, multiplying (\ref{eq-density-1}) by powers of $N_{U_1}(u)^{-1}$ proves the density of $\bigcup^{}_{m\geq 1}R_m$ in $V$.  
\end{proof}

Combining this density lemma with the tower of finite \'etale torsors, we now upgrade the ring-level perfectoidness of Proposition \ref{prop-ring-perf} into a geometric one. 
\begin{prop}\label{prop-aff-quot}
	The natural morphism $\cV\rightarrow \cR_1$ is a pro-\'etale $U_1$-torsor. Consequently, we have  isomorphisms of $v$-sheaves $$\cW/\underline{\OO^{\times}_{K}}\cong \cV/\underline{U_1}\cong \cR_1.$$ In particular,  $\cW/\underline{\OO^{\times}_{K}}$ is affinoid perfectoid. 
\end{prop}
\begin{proof}
	Let $t=N_{U_1}(u)$. By part (i) of Corollary \ref{cor-norm-prop}, $t$ is a topologically nilpotent unit in $R_m$ for every $m\geq 1$. Thus for every $s\geq 0$, we have 
	\[
	R_m\cap t^{s}V^{+}=t^{s}R_m^{+},
	\] 
	hence $V^{+}$ is identified with the $t$-adic completion of $\dirlim_{m}R_{m}^{+}$ by Lemma \ref{lem-density}. Also since  $V=V^{+}[1/t]$ and each $\cR_m$ is affinoid perfectoid (Proposition \ref{prop-ring-perf}), by \cite[Proposition~6.5]{Scholze-Diamond} we get an isomorphism of perfectoid spaces 
	\[
	\cV\cong \invlim_{m}\cR_m.
	\] 
	By  Proposition \ref{prop-torsor}, such $\cR_m$ form a tower of finite \'etale torsors and we have 
	\begin{align*}
		\cV\times _{\cR_1}\cV &\cong \invlim_{m}(\cR_m\times_{\cR_1}\cR_m)\\
		&\cong \invlim_{m}(\underline{U_1/U_m}\times \cR_m)\\
		&\cong \underline{U_1}\times \cV.
	\end{align*}
	Therefore, $\cV\rightarrow \cR_1$ is a pro-\'etale $U_1$-torsor, hence we get 
	\[
		\cR_1\cong \Coeq(\underline{U_1}\times \cV\rightrightarrows \cV)=\cV/\underline{U_1}. 
	\] 
	Finally, part (ii) of Proposition \ref{prop-tame-quotient} gives  $\cW/\underline{\mu_{q-1}}\cong \cV$ where the $\mu_{q-1}$-action commutes with the $U_1$-action. 
	Thus we  get $\cW/\underline{\OO^{\times}_{K}}\cong \cR_1$, which is affinoid perfectoid.  
\end{proof}

\begin{thm}\label{thm-Gr1-perf}
	For any  $0<d<n$, the  $v$-sheaf quotient ${\bB^{\varphi^{n}=\pi^{d}}\bs \{0\}}/\underline{\OO^{\times}_K}$ is representable by a perfectoid space.
\end{thm}
\begin{proof}
	By Proposition \ref{prop-stable}, the  affinoid perfectoid cover $\bigcup^{d-1}_{i=0}\cW_i$ of $\bB^{\varphi^{n}=\pi^{d}}\bs \{0\}$ is $\OO^{\times}_{K}$-stable, thus we get a commutative diagram
	\[
	\begin{tikzcd}
		\cW_i\arrow[r]\arrow[d] & \bB^{\varphi^{n}=\pi^{d}}\bs \{0\}\arrow[d]\\
		\cW_i/\underline{\OO^{\times}_{K}}\arrow[r] & (\bB^{\varphi^{n}=\pi^{d}}\bs \{0\})/\underline{\OO^{\times}_{K}}
	\end{tikzcd}
	\] 
	As the upper horizontal map is an open immersion and the quotient map on the right is $v$-surjective, the lower horizontal map is also an open immersion. As such $\cW_i/\underline{\OO^{\times}_{K}}$ cover $(\bB^{\varphi^{n}=\pi^{d}}\bs \{0\})/\underline{\OO^{\times}_{K}}$ and each of them is perfectoid by Proposition \ref{prop-aff-quot}, $(\bB^{\varphi^{n}=\pi^{d}}\bs \{0\})/\underline{\OO^{\times}_{K}}$ is covered by open perfectoid subspaces, hence  perfectoid. 
\end{proof}

\begin{cor}\label{cor-Gr1-perf}
	For any $0<d<n$, the Grassmannian $\Gr_{1,C}(\bB^{\varphi^{n}=\pi^{d}})$ is representable by a perfectoid space. 
\end{cor}
\begin{proof}
By definition, we have 
	\begin{align*}
		\Gr_{1,C}(\bB^{\varphi^{n}=\pi^{d}})&=((\bB^{\varphi^{n}=\pi^{d}}\bs \{0\})\times _{\Spd \Fq}\Spd C)/\underline{K^{\times}}\\
		&= ((\bB^{\varphi^{n}=\pi^{d}}\bs \{0\})/\underline{\OO^{\times}_{K}}\times_{\Spd \Fq} \Spd C)/(\pi\times \id)^{\Z}.
	\end{align*}
	Using  Theorem \ref{thm-Gr1-perf}, the product before taking the quotient by $\pi\times \id$ is perfectoid. Since the action of $\pi\times \id$ is properly discontinuous on the product, the quotient is perfectoid. 
\end{proof}


\begin{rem}\label{rem-new-JLH}
	Recall from section \ref{section-1.3}, for $S\in \Perf_{C}$, the moduli space $W_{n,1}$ parametrizes surjections of the form $\OO(1/n)\surj \cE$ over $X_{S^{\tilt}}$ with $\cE$ isomorphic to $\OO(1)$ at every geometric point. By duality, a surjection $\OO(1/n)\surj \OO(1)$ is equivalent to  a stably injective morphism $\OO\hookrightarrow \OO(\frac{n-1}{n})$; also notice both $\OO(1)$ and $\OO$ have automorphism group $K^{\times}$. Therefore, we have 
	\begin{align*}
		W_{n,1}&\cong \left(\Inj\left(\OO, \OO\left(\frac{n-1}{n}\right)\right)\times _{\Spd \Fq}\Spd C\right)/\underline{K^{\times}}\\
		&=\Gr_{1,C}(\bB^{\varphi^{n}=\pi^{n-1}}).
	\end{align*}
Thus Corollary \ref{cor-Gr1-perf} gives an alternative local proof of Theorem \ref{thm-JLHmain}. 
\end{rem}


\section{Pl\"{u}cker morphisms}\label{section-4}
For classical Grassmannians of a vector space $E$, there is the Pl\"{u}cker map which embeds  the $k$-Grassmannian of $E$ into   the projective space of the $k$-th exterior power of $E$. In this section, we study the analogous  Pl\"{u}cker morphisms between Grassmannians of absolute BC spaces. As BC spaces are characterized by both their rank and degree, the latter will also play a role in describing the image of these Pl\"{u}cker morphisms.
\subsection{The classical Pl\"{u}cker embedding}
To motivate our later constructions, we give a brief exposition on classical Pl\"{u}cker map and Pl\"{u}cker relations. See \cite[\S~III.11-13]{Bourbaki-algebra} for more details on this material.

Let $E$ be an $n$-dimensional vector space over a field $K$ and let $0<k<n$. For a $k$-dimensional subspace $W$ of $E$, the  \textit{Pl\"{u}cker map} is given by
\begin{align*}
	\Pl: \Gr_k(E)&\rightarrow \P(\bigwedge^{k}E)\\
	W&\mapsto \bigwedge^{k}W.
\end{align*}
Notice here that changing the basis of $W$ would change the wedge product $\bigwedge^{k}W$ by a scalar. 
A vector in $\P(\bigwedge^{k}E)$ is called \textit{decomposable} if it is a single  wedge product of $k$ vectors in $E$, and the image of the Pl\"{u}cker map is exactly this decomposable locus inside  the projective space.

We now  give a more algebraic description of the decomposable vectors. Denote by $E^{\ast}$ the dual vector space of $E$ and let 
\[
\langle -,-\rangle_{r}: \bigwedge^{r}E^{\ast}\otimes \bigwedge^{r}E\rightarrow K
\] 
be the perfect pairing induced by the evaluation map. The adjoint of this map then gives the \textit{contraction map}
\[
c: \bigwedge^{k-1}E^{\ast}\otimes \bigwedge^{k}E\rightarrow E
\] 
which satisfies  $\langle \beta^{\ast}, c(\alpha ^{\ast}, w)\rangle_1=\langle \alpha^{\ast}\wedge \beta^{\ast}, w\rangle_k$. More explicitly, given $\alpha ^{\ast}\in \bigwedge^{k-1}E^{\ast}$ and $w_1\wedge\cdots \wedge w_k\in \bigwedge^{k}E$, we have 
\begin{equation}\label{eq-contra}
	c(\alpha ^{\ast}, w_1\wedge\cdots \wedge w_k)=\sum^{k}_{j=1}(-1)^{k-j}\langle \alpha ^{\ast}, w_1\wedge\cdots\wedge \wh{w_j}\wedge\cdots \wedge w_k\rangle_{k-1} w_j.
\end{equation}
For a nonzero vector $w\in \bigwedge^{k}E$, consider its associated contraction map and exterior multiplication map:
\begin{alignat}{2}
	&c_w: \bigwedge^{k-1}E^{\ast}\rightarrow E, \quad\quad && \alpha ^{\ast}\mapsto c(\alpha ^{\ast}, w)\\
	&m_w: E\rightarrow \bigwedge^{k+1}E, \quad\quad && v\mapsto v\wedge w. 
\end{alignat}
If  $w$ belongs to the image of the Pl\"{u}cker map, these two maps describe the $k$-plane $W$ associated to $w$  in two directions: the image of the contraction map $c_w$ recovers the plane  $W\subseteq E$, and  the exterior multiplication map $m_{w}$ vanishes exactly on $W$. Therefore, we have:

\begin{prop}[{\cite[Proposition~III.13.16]{Bourbaki-algebra}}]\label{prop-bourbaki}
A nonzero vector $w\in \bigwedge^{k}E$ is decomposable if and only if $m_w\circ c_w=0$.
\end{prop}

The image of the Pl\"{u}cker map can  also be  described via  \textit{Pl\"{u}cker relations}. To recover such quadratic equations, consider
\begin{align*}
	Q: \bigwedge^{k}E&\rightarrow \Hom_{K}(\bigwedge^{k-1}E^{\ast}, \bigwedge^{k+1}E)\\
	w&\mapsto m_w\circ c_w.
\end{align*}
Let $\{e_i\}$ be a basis of $E$ with dual basis $\{e_i^{\ast}\}$. For a vector  $w\in \bigwedge^{k}E$, write $w=\sum^{}_{I}p_Ie_I$ for ordered $k$-tuples $I\subseteq \{1,\ldots, n\}$, where $e_I$ are the basis vectors of $\bigwedge^{k}E$ and $p_I$ are the Pl\"{u}cker coordinates.
Then to have  $Q(w)=0$, we need $Q(w)(e_J^{\ast})=0$ for every dual basis vector $e_J^{\ast}$ of $\bigwedge^{k-1}E^{\ast}$, which means the coordinate components of $Q(w)$ all vanish. This gives the Pl\"{u}cker relations.

\begin{exmp}
Let  $E$ be a 4-dimensional vector space with basis $e_1,\ldots, e_4$ and consider the Pl\"{u}cker map
\[
	\Pl: \Gr_2(E)\rightarrow \P(\bigwedge^{2}E).
\] 
	By the discussion above, for  $w=\sum^{}_{I}p_Ie_I\in \bigwedge^{2}E$, the class $[w]$ is in the image of $\Pl$ if and only if $Q(w)(e_i^{\ast})=0$ for $i=1,2,3,4$. 
	From  (\ref{eq-contra}), we get  
\[
	c_w(e_1^{\ast})=p_{12}e_2+p_{13}e_3+p_{14}e_4.
\] 
	Taking the wedge product with $w$ then gives $(p_{12}p_{34}-p_{13}p_{24}+p_{14}p_{23})e_{234}$, which induces the well-known Pl\"{u}cker relation
\[
p_{12}p_{34}-p_{13}p_{24}+p_{14}p_{23}=0.
\] 
Replacing $e_1^{\ast}$ by other dual basis vectors gives the same quadratic relation in this example. 
\end{exmp}

\subsection{The decomposable locus}
We now transition back to the setting of vector bundles on the Fargues--Fontaine curve, and define the decomposable locus of Grassmannians of absolute Banach--Colmez spaces. 

Let $S\in \Perf_{\Fq}$ and let $\cE$ be a vector bundle on $X_{S}$ of rank $n$. We write $\cE^{\vee}$ for its dual bundle. For any integer $0< k< n$, we have the contraction map 
 \[
c: \bigwedge^{k-1}\cE^{\vee}\otimes \bigwedge^{k}\cE\rightarrow \cE
\] 
which is adjoint to the perfect pairing $\langle-,-\rangle: \bigwedge^{k}\cE^{\vee}\otimes \bigwedge^{k}\cE\rightarrow \OO_{X_S}$. By construction, we have $\langle \beta, c(\alpha ,w)\rangle=\langle \alpha \wedge \beta, w\rangle$ on local sections. By further composing the contraction map  $c$ with  the exterior product map $m: \cE\otimes \bigwedge^{k}\cE\rightarrow \bigwedge^{k+1}\cE$, we get a morphism 
\[
	m\circ (c\otimes \id):  \bigwedge^{k-1}\cE^{\vee}\otimes \bigwedge^{k}\cE\otimes \bigwedge^{k}\cE\rightarrow \bigwedge^{k+1}\cE.
\] 
Thus on local sections, we have $m\circ (c\otimes \id)(\alpha , w_1,w_2)=c(\alpha, w_1)\wedge w_2$. All these constructions commute with base change in $S$.

Now take $\cE$ to be the semistable vector bundle $\cE_{n,d}$ of rank $n$ and degree $d$. Its  $k$-th wedge power gives  $\bigwedge^{k}\cE_{n,d}= \cE_{N,D}$ where $\cE_{N,D}$ is the semistable vector bundle over $X_{S}$ of rank $N=\binom{n}{k}$ and degree $D=d\cdot \binom{n-1}{k-1}$. 

Consider the $v$-sheaf  $\Inj(\OO, \bigwedge^{k}\cE_{n,d})/\underline{K^{\times}}$.  An $S$-point of this sheaf is represented by a stably injective morphism $u: \cL\rightarrow \bigwedge^{k}\cE_{n,d}$ where $\cL$ is a line bundle over $X_{S}$ fibrewise isomorphic to $\OO$. For such a pair $(u,\cL)$, we let 
\begin{align*}
	&c_u:=c\circ (\id\otimes u): \bigwedge^{k-1}\cE^{\vee}_{n,d}\otimes \cL\rightarrow \cE_{n,d}\\
	&m_u:= m\circ (\id\otimes u): \cE_{n,d}\otimes \cL\rightarrow \bigwedge^{k+1}\cE_{n,d}.
\end{align*}
Their composition then gives a morphism 
\[
	q_u:=m_u\circ (c_u\otimes \id): \bigwedge^{k-1}\cE^{\vee}_{n,d}\otimes \cL^{\otimes 2}\rightarrow \bigwedge^{k+1}\cE_{n,d}.
\] 
Let also $\wt{m}_u: \cE_{n,d}\rightarrow \cL^{\vee}\otimes \bigwedge^{k+1}\cE_{n,d}$ be the adjoint morphism of $m_u$, then we have:

\begin{lem}\label{lem-qu=0}
	The morphism $q_u$ is 0 if and only if $\wt{m}_u\circ c_u$ is 0. 
\end{lem}
\begin{proof}
Since $\cL$ is a line bundle, adjunction gives an isomorphism 
\[
	\Hom(\bigwedge^{k-1}\cE_{n,d}^{\vee}\otimes \cL, \cL^{\vee}\otimes \bigwedge^{k+1}\cE_{n,d})\cong \Hom_{\OO_{X_S}}(\bigwedge^{k-1}\cE^{\vee}_{n,d}\otimes \cL\otimes \cL, \bigwedge^{k+1}\cE_{n,d}),
\]
	under which  $\wt{m}_u\circ c_u$ corresponds to $q_u$. This proves the claim. 
\end{proof}

 Similar to the classical case discussed previously, we now use $q_u$ to define the decomposable locus of $\Inj(\OO, \bigwedge^{k}\cE_{n,d})/\underline{K^{\times}}$.

\begin{defn}\label{defn-Dec}
	For integers $1\leq k\leq n$ and $d\in \N^{\ast}$, we denote by $\Dec^{k}(\cE_{n,d})$ the sheaf on $\Perf_{\Fq}$ which assigns to $S$ the set of $S$-points  $[(u,\cL)]$ of $\Inj(\OO, \bigwedge^{k}\cE_{n,d})/\ul{K}^{\times}$ such that $q_u=0$.  We call $\Dec^{k}(\cE_{n,d})$ the \textit{decomposable locus} of $\Inj(\OO, \bigwedge^{k}\cE_{n,d})/\ul{K}^{\times}$. 
\end{defn}

One can easily check that the condition $q_u=0$ does not depend on the representative $(u,\cL)$: if $\varphi:(u,\cL)\rightarrow (u',\cL')$ is an isomorphism between two such pairs, then $u'\circ \varphi=u$, hence $q_{u'}\circ (\id\otimes \varphi^{\otimes 2})=q_u$. 
If one picks  an open subset of $X_S$ and a frame of the vector bundle $\cE_{n,d}$, then locally the condition $q_u=0$ recovers the Pl\"{u}cker relations on such coordinates. Since $\Inj(\OO, \bigwedge^{k}\cE_{n,d})/\underline{K^{\times}}=\Gr_{1}(\bB^{\varphi^{N}=\pi^{D}})$ for some $N,D\in \N^{\ast}$, $\Dec^{k}(\cE_{n,d})$ defines a $k$-decomposable locus of the Grassmannian.

\begin{lem}\label{lem-vsub}
	The decomposable locus $\Dec^{k}(\cE_{n,d})$ is a $v$-subsheaf of $\Inj(\OO, \bigwedge^{k}\cE_{n,d})/\ul{K}^{\times}$. 
\end{lem}
\begin{proof}
The maps between vector bundles related to the construction of $q_u$ all commute with base change. Hence for a $v$-cover $T\rightarrow S$, we have $q_{u_T}=q_u|_{X_T}$. Moreover, vanishing of a morphism of vector bundles is local for the $v$-topology. Therefore, the condition $q_u=0$ satisfies $v$-descent. 
\end{proof}


\subsection{Pl\"{u}cker morphisms}
In this section, we construct the Pl\"{u}cker morphism between Grassmannians, and identify its image with the decomposable locus defined above.

For $x\in \left|X_S\right|$ and $\cE$ a vector bundle on $X_S$, write $\cE(x)=\cE\otimes_{\OO_{X_S}}\kappa(x)$ for its fibre at $x$. For a morphism $f$ between vector bundles on $X_S$, let $f(x)$ denote the corresponding map on the fibres.  
\begin{lem}\label{lem-const-rank}
	Let $f: \cE\rightarrow \cF$ be a morphism of vector bundles on $X_S$. If the function $x\mapsto \rank_{\kappa(x)}f(x)$ is locally constant,  $\ker(f), \im(f), \coker(f)$ are vector bundles. 
\end{lem}
\begin{proof}
	The assertion is local on $X_S$. Since $X_S$ is stably uniform, we may work on an affinoid open $U=\Spa(R,R^{+})\subseteq X_S$ with $R$ uniform, on which $\cE$ and $\cF$ are free  and $f(x)$ has constant rank $r$. Then $f$ comes from a map $f_R: P\rightarrow Q$ between finite free $R$-modules. Let $M=\coker(f_R)$. For every $x\in U$, we have  $M\otimes _R\kappa(x)\cong \coker f(x)$. Its dimension $\dim \cF(x)-r$ is locally constant on $U$. Hence $M$ is finite projective by \cite[Proposition~2.8.4]{KL-relative}. It follows that $Q\rightarrow M$ splits, so  $N=\im(f_R)$ is finite projective. Thus $P\rightarrow N$ also splits, so $K=\ker(f_R)$ is finite projective. Tensoring these split exact sequences with $\OO_U$ identifies $\ker(f)|_U, \im(f)|_U, \coker(f)|_U$ with the vector bundles associated to $K, N, M$ respectively.  
\end{proof}

\begin{lem}\label{lem-coker}
Let $n,d,k\in \N^{\ast}$ and assume $kd<n$. 
\begin{renumerate}
	\item For any  injective morphism $i: \cF \rightarrow \cE_{n,d}$ over $X_{S}$ with $\cF$ fibrewise isomorphic to $\OO^{k}$, $\coker(i)$ is a vector bundle on  $X_S$. 
	\item For any  injective morphism $u: \cL\rightarrow \bigwedge^{k}\cE_{n,d}$ over $X_S$ with $\cL$ fibrewise isomorphic to $\OO$, $\coker(u)$ is a vector bundle on  $X_S$. 
\end{renumerate}
\end{lem}
\begin{proof}
	By \cite[Remark~A.2.3]{JLH2021-1}, it suffices to show that $\cF$ and $\cL$ are saturated  when $S=\Spa(F,F^{+})$ is a geometric point. For (i), we have 
	\[
		\frac{\deg \cF+1}{\rank  \cF}=\frac{1}{k}>\frac{d}{n}=\mu(\cE_{n,d}).
	\] 
	Similarly for (ii), we have 
	\[
		\frac{\deg \cL +1}{\rank \cL}=1>\mu(\bigwedge^{k}\cE_{n,d})=\frac{kd}{n}.
	\] 
	By \cite[Lemma~A.2.1]{JLH2021-1}, such  inequalities imply both $\cF$ and $\cL$ are saturated, hence the corresponding quotients are vector bundles on  $X_S$. 
\end{proof}

Let $i: \cF\rightarrow \cE_{n,d}$ be a stably injective morphism over $X_S$ with $\cF$ fibrewise isomorphic to $\OO^{k}$. Taking its determinant gives $\wedge^{k}i: \det \cF\rightarrow \bigwedge^{k}\cE_{n,d}$, which is also  stably injective. Moreover, given two such pairs $(i,\cF)$ and $(i', \cF')$ together with an  isomorphism $\varphi: (i, \cF)\rightarrow (i',\cF')$, we have  $i'\circ \varphi=i$, hence $(\wedge^{k}i')\circ \det \varphi= \wedge^{k}i$. 
Thus $\det(\varphi)$ gives an isomorphism between  $(\wedge^{i}, \det\cF)$ and $(\wedge^{k}i', \det \cF')$. Therefore, we have a well-defined morphism of $v$-sheaves
\begin{align*}
	\Pl: \Inj(\OO^{k}, \cE_{n,d})/\underline{\GL_k(K)} &\rightarrow \Inj(\OO, \bigwedge^{k}\cE_{n,d})/\underline{K^{\times}}\\
	[(i, \cF)] &\mapsto [(\wedge^{k}i, \det \cF)]
\end{align*}
which we call the \textit{Pl\"{u}cker morphism}.

Our motivation for considering such a map comes from the special case when  $d=1$ and $k=n-1$, where the quotient on the right hand side is isomorphic to $\Inj(\OO, \OO(\frac{n-1}{n}))/\underline{K^{\times}}$ and the left one is isomorphic to $\Surj(\OO(1/n), \OO(1))/\underline{K^{\times}}$. Duality then implies  that $\Pl$ is an isomorphism (cf. Remark \ref{rem-new-JLH}). In the general case, however, such a Pl\"{u}cker morphism need not be surjective. We will show later that it nevertheless induces an isomorphism onto the decomposable locus.

\begin{lem}\label{lem-factor}
The Pl\"{u}cker morphism
\[
	\Pl: \Inj(\OO^{k}, \cE_{n,d})/\underline{\GL_k(K)}\rightarrow \Inj(\OO, \bigwedge^{k}\cE_{n,d})/\underline{K^{\times}}
\] 
	factors through the decomposable locus $\Dec^{k}(\cE_{n,d})$. 
\end{lem}
\begin{proof}
	Let $(i, \cF)$ represent an $S$-point of $\Inj(\OO^{k},\cE_{n,d})/\underline{\GL_k(K)}$. We choose local generators $v_1,\ldots,v_k$ of $\cF$. Under the Pl\"{u}cker morphism, the corresponding local generator of $\det \cF$ is given by  $z=i(v_1)\wedge\cdots \wedge i(v_k)$. By definition of the contraction map, for every local section $\alpha $ of $\bigwedge^{k-1}\cE^{\vee}_{n,d}$ the contraction $c(\alpha, z)$ belongs to the submodule generated by $i(v_1),\ldots,i(v_k)$. Therefore, we get $c(\alpha ,z)\wedge z=0$ and $q_{\wedge^{k}i}=0$. Hence the image of $\Pl$ lies in $\Dec^{k}(\cE_{n,d})$.
\end{proof}

\begin{prop}\label{prop-reconstruct}
	Assume  $kd<n$. For every pair $(u,\cL)$ representing an $S$-point of $\Dec^{k}(\cE_{n,d})$, there is a unique rank $k$ subbundle $\cV_u\subseteq \cE_{n,d}$ such that
	\[
		\det \cV_u=u(\cL)\subseteq \bigwedge^{k}\cE_{n,d}.
	\] 
\end{prop}
\begin{proof}
	By part (ii) of Lemma \ref{lem-coker}, $u(\cL)\subseteq \bigwedge^{k}\cE_{n,d}$ is a line subbundle. Since $q_u=0$,  Proposition \ref{prop-bourbaki} and Lemma \ref{lem-qu=0} show that $u(x)(\cL(x))$ is decomposable for every $x\in \left|X_S\right|$, hence it is a single wedge product of $k$ vectors. Therefore, $\ker(\wt{m}_u(x))$ has dimension $k$ and $\wt{m}_u(x)$ has rank $n-k$. Lemma \ref{lem-const-rank} then  gives a rank $k$ subbundle 
	\[
		\cV_u:=\ker(\wt{m}_u)\subseteq \cE_{n,d}.
	\] 
	Specifically, we get  $\cV_u(x)=\ker(\wt{m}_u(x))$, hence  $\det(\cV_u(x))=u(x)(\cL(x))$. Therefore, the composition 
	\[
		\cL\xrightarrow{u}\bigwedge^{k}\cE_{n,d}\rightarrow (\bigwedge^{k}\cE_{n,d})/\det \cV_u
	\] 
	is zero at every fibre, hence the composition itself is also zero by Lemma \ref{lem-const-rank}. Thus the morphism $u: \cL\rightarrow \bigwedge^{k}\cE_{n,d}$ factors through $\det \cV_u$ and gives $u(\cL)=\det \cV_u$. 

	Suppose  $\cV\subseteq \cE_{n,d}$ is another rank $k$ subbundle with $\det \cV=u(\cL)$. We then have  $\cV \wedge u(\cL)=0$, which gives  $\cW\subseteq \ker(\wt{m}_u)=\cV_u$. Since $\det \cV=u(\cL)=\det \cV_u$, we get  $\cV=\cV_u$, which  proves the uniqueness part. 
\end{proof}

\begin{lem}\label{lem-deg-0}
	Let $\cV_u$ be the subbundle of Proposition \ref{prop-reconstruct}. For every geometric point $\ol{x}$ of $S$, we have $\cV_{u,\ol{x}}\cong \OO^{k}_{X_{\ol{x}}}$. 
\end{lem}
\begin{proof}
	Since $\det(\cV_u)=\cL$, the vector bundle  $\cV_u$ has degree 0. Let $\ol{x}$ be any geometric point of $S$. Suppose the maximal HN piece of $\cV_{u,\ol{x}}$ has positive degree and rank $r\leq k$, then its slope is at least $1/r$. Since $\cE_{n,d}$ is semistable, we get $1/r\leq d/n$, and the latter is $<1/k$ by assumption; this contradicts the constraint $r\leq k$. Thus all the HN slopes of $\cV_{u,\ol{x}}$ are non-positive, which implies $\cV_{u,\ol{x}}\cong \OO^{k}_{X_{\ol{x}}}$. 
\end{proof}

\begin{thm}\label{thm-iso-dec}
Suppose  $kd<n$. The Pl\"{u}cker morphism induces an isomorphism
\[
	\Pl: \Gr_{k}(\bB^{\varphi^{n}=\pi^{d}})\xrightarrow{\sim} \Dec^{k}(\cE_{n,d}). 
\] 
\end{thm}
\begin{proof}
	It suffices to show that for every $S\in \Perf_{\Fq}$, $\Pl(S)$ is bijective. Surjectivity is given by Proposition \ref {prop-reconstruct} and Lemma \ref{lem-deg-0}. If two $S$-points $(i,\cF)$ and $(i',\cF')$ of the Grassmannian have the same image under $\Pl(S)$, then we get $\det\cF\cong \det \cF'$, and the images $(\wedge^{k}i)(\cF), (\wedge^{k}i')(\cF')$ identify the same line subbundle of $\bigwedge^{k}\cE_{n,d}$. By the uniqueness part of Proposition \ref{prop-reconstruct}, $i(\cF)$ and $i'(\cF')$ then identify the same vector subbundle of $\cE_{n,d}$, hence they represent the same isomorphism class in the Grassmannian. 
\end{proof}

We now show the decomposable locus  is Zariski closed inside $\Inj(\OO, \bigwedge^{k}\cE_{n,d})/\underline{K^{\times}}$. First we recall the following lemma of Fargues--Scholze: 
\begin{lem}[{\cite[Lemma~IV.4.23]{FS2024-1}}]\label{lem-FS}
	Let $S=\Spa (R,R^+)\in \Perf_{\Fq}$ be an affinoid perfectoid space with pseudouniformizer $\varpi$ and $I\subseteq (0,\infty)$ a compact interval with rational ends. Then for every ideal $J\subseteq \OO(Y_{S,I})$, there exists a unique Zariski closed perfectoid subspace $S(J)$ of $S$ such that for every perfectoid space $T\rightarrow S$, its structural morphism factors through $S(J)$ if and only if the image of $J$ is 0  along  $\OO(Y_{S,I})\rightarrow \OO(Y_{T,I})$. 
\end{lem}

\begin{prop}\label{prop-Zariski-closed}
	The decomposable locus $$\Dec^{k}(\cE_{n,d})\subseteq \Inj(\OO, \bigwedge^{k}\cE_{n,d})/\underline{K^{\times}}$$ is Zariski closed.
\end{prop}
\begin{proof}
	Let us simply write $\cI$ for  $\Inj(\OO, \bigwedge^{k}\cE_{n,d})/\underline{K^{\times}}$. For any affinoid perfectoid $S=\Spa(R,R^{+})\in \Perf_{\Fq}$, we want to  show the fibre product 
	\[
	\begin{tikzcd}
		S\times _{\cI}\Dec^{k}(\cE_{n,d})\arrow[d]\arrow[r] & S\arrow[d]\\
		\Dec^{k}(\cE_{n,d})\arrow[r] & \cI
	\end{tikzcd}
	\] 
	is Zariski closed in $S$.  Since the decomposable locus is defined by the condition $q_u=0$, this fibre product parametrizes exactly those $T\rightarrow S$ such that the pullback $(q_u)_T$ is 0.  

	For $(u, \cL)$ representing an $S$-point of $\cI$, the associated map  $q_u$ gives a global section of 
	\[
		\cH:=\HHom_{\OO_{X_S}}(\bigwedge^{k-1}\cE_{n,d}^{\vee}\otimes \cL^{\otimes 2}, \bigwedge^{k+1}\cE_{n,d})
	\] 
	over $X_S$. We pick the  interval $I=[1,q]\subset (0,\infty)$ and consider $Y_{S,I}$.  Also denote by $B_{S,I}=\OO(Y_{S,I})$ and by $\pi: Y_{S,I}\rightarrow X_S$ the covering map. Since $Y_{S,I}$ is sousperfectoid,  \cite[Theorem~5.2.8, Proposition~6.3.4]{SW-Berkeley} imply that $H:=\Gamma(Y_{S,I},\pi^{\ast}\cH)$ is a finite projective $B_{S,I}$-module. We let $h_u$ be the pullback of $q_u$ along the cover $\pi$. Then evaluation at  $h_u$ gives a $B_{S,I}$-linear map  $H^{\vee}\rightarrow B_{S,I}$, and 
	\[
		J:=\im(H^{\vee}\rightarrow B_{S,I})
	\] 
	gives  an ideal of $B_{S,I}$. As $H$ is finite projective, $J=0$ is equivalent to $h_u=0$.

	By Lemma \ref{lem-FS}, the ideal $J$ gives  a  Zariski closed perfectoid subspace, say $S^{\dec}$,  of $S$ such that a perfectoid space $T\rightarrow S$ factors through $S^{\dec}$ if and only if $J\cdot B_{T,I}=0$. Assume $T$ is affinoid perfectoid. Since  $H$ is finite projective, $J\cdot B_{T,I}=0$ is equivalent to $h_u\otimes 1=0$ in $H\otimes _{B_{S,I}}B_{T,I}$; in the meantime, $h_u\otimes 1$ is exactly the pullback of $(q_u)_T$ to $Y_{T,I}$. Since $I=[1,q]$, Frobenius translates of $Y_{T,I}$ cover $Y_{T}$, hence the cover $Y_{T}\rightarrow X_{T}$ is Frobenius equivariant. Therefore, $J\cdot B_{T,I}=0$ is equivalent to $(q_u)_T=0$.  This  also extends to arbitrary perfectoid $T\rightarrow S$ by taking affinoid open cover of $T$. Thus by the Yoneda lemma, we get $S^{\dec}\cong S\times_{\cI}\Dec^{k}(\cE_{n,d})$, hence the assertion.  
\end{proof}

For an absolute BC space of slope $<1/k$, the Pl\"{u}cker morphism therefore reduces the perfectoidness of its $k$-Grassmannian to the case of a projectivized absolute BC space. Combining this with what we established in section \ref{section-3} gives:
\begin{thm}\label{thm-dk<n}
	The Grassmannian $\Gr_{k,C}(\bB^{\varphi^{n}=\pi^{d}})$ is perfectoid if $kd<n$. 
\end{thm}
\begin{proof}
	Let  $N=\binom{n}{k}$ and $D=d\cdot \binom{n-1}{k-1}$. Then we have 
	\[
		(\Inj(\OO, \bigwedge^{k}\cE_{n,d})\times_{\Spd \Fq}\Spd C)/\underline{K^{\times}}\cong \Gr_{1,C}(\bB^{\varphi^{N}=\pi^{D}}).
	\] 
	The condition $kd<n$ implies  $D<N$.  Hence  Corollary \ref{cor-Gr1-perf} shows   $\Gr_{1,C}(\bB^{\varphi^{N}=\pi^{D}})$ is perfectoid. By Theorem \ref{thm-iso-dec} and Proposition \ref{prop-Zariski-closed}, the Pl\"{u}cker morphism
	\[
		\Pl: \Gr_{k,C}(\bB^{\varphi^{n}=\pi^{d}})\rightarrow \Gr_{1,C}(\bB^{\varphi^{N}=\pi^{D}})
	\] 
	is a Zariski closed immersion, thus $\Gr_{k,C}(\bB^{\varphi^{n}=\pi^{d}})$ is perfectoid. 
\end{proof}

From the statements leading  to Theorem \ref{thm-dk<n}, one can see that the assumption  $kd<n$ is crucial. For $k=1$,  the condition $d<n$ was essential for the convergence of the orbit norms. For higher Grassmannians, the Pl\"{u}cker morphism will not be an isomorphism onto the decomposable locus without assuming $kd<n$.

\section{Non-perfectoid Grassmannians}
In this section, we show the previous assumption of $kd<n$ is necessary by showing the $k$-Grassmannian of an absolute BC space  is not perfectoid whenever its slope is $\geq 1/k$.  

We start with the  simple yet foundational example of $\Gr_{1,C}(\bB^{\varphi=\pi})$. Recall that 
\[
	\Inj(\OO, \OO(1))/\underline{K^{\times}}\cong (\bB^{\varphi=\pi}\bs \{0\})/\underline{K^{\times}}.
\] 
The right hand side is identified with $\Div^{1}$ by \cite{Fargues2020-1}, and the same argument as in Theorem \ref{thm-Hansen} shows the left hand side is isomorphic to $\Spd K/\varphi_K^{\Z}$. Hence we have 
\[
	\Gr_{1,C}(\bB^{\varphi=\pi})=(\Spd K\times _{\Spd \Fq}\Spd C)/(\varphi_K\times \id)^{\Z}\cong \Div^{1}_{C}.
\] 
That this is not perfectoid is presumably well known to the experts; see e.g. \cite[p.2]{Weinstein2017-1}. We nevertheless include a proof here.

\begin{prop}\label{prop-div}
The diamond $\Div^{1}_{C}$ is not representable by a perfectoid space.  
\end{prop}
\begin{proof}
	Write $\Div^{1}_{C}$ as $(\Spd K\times_{\Spd \Fq}\Spd C)/(\varphi_{K}\times \id)^{\Z}$. By \cite[Proposition~II.1.17]{FS2024-1}, there is a natural isomorphism  $Y^{\diamond}_{C^{\tilt}}\cong \Spd K\times_{\Spd \Fq}\Spd C$, which also induces $X^{\diamond}_{C^{\tilt}}\cong (\Spd K\times_{\Spd \Fq}\Spd C)/(\id \times \varphi_{C^{\tilt}})^{\Z}$.\footnote{Notice here that $\Div^{1}_C$ and $X^{\diamond}_{C^{\tilt}}$ correspond to different Frobenius quotients of $Y^{\diamond}_{C^{\tilt}}$: one is the Frobenius on $\Spd K$ and the other is on $\Spd C$.}  
	For any rational interval $I\subseteq (0,\infty)$, consider the affinoid open  subset $Y_{C^{\tilt},I}=\Spa(B_I, B^{+}_{I})$ of $Y_{C^{\tilt}}$. By \cite[Proposition~II.1.16]{FS2024-1}, there is a radius map $\rad: \left|Y_{C^{\tilt}}\right|\rightarrow (0,\infty)$ such that $\rad\circ \varphi_{C^{\tilt}}=q\cdot \rad$ and $Y_{C^{\tilt},I}\subset \rad^{-1}(I)$. As the product of the two Frobenius maps acts trivially on the topological space $\left|Y_{C^{\tilt}}\right|$, we get  $\varphi_K^{n}(Y_{C^{\tilt},I})\subset \rad ^{-1}(q^{-n} I)$. Thus if we take $I=[a,b]$ with $a<b<qa$, the $\varphi_K$-translates of $I$ are disjoint, hence the same for the  $\varphi_K$-translates of $Y^{\diamond}_{C^{\tilt},I}$. Hence we get an open immersion 
	\[
	Y^{\diamond}_{C^{\tilt},I}\rightarrow \Div^{1}_{C},
	\] 
	and it is enough to show $Y^{\diamond}_{C^{\tilt},I}$ is not perfectoid. By  \cite[Remark~11.2 and \S13]{HK-sheaf}, we are further reduced to showing  $Y_{C^{\tilt},I}$ is not perfectoid. 
	 
	Assume $Y_{C^{\tilt},I}=\Spa(B_I,B_I^{+})$ is perfectoid. Since $Y_{C^{\tilt},I}$ is affinoid and thus quasicompact, it admits an open cover by finitely many affinoid perfectoid subspaces $U=\Spa(R,R^{+})$; specifically, each $R$ is perfectoid.\footnote{Each $U$ is picked to be affinoid perfectoid, which by definition means $R$ is perfectoid. On the other hand, the assumption of $Y_{C^{\tilt},I}$ being perfectoid does not imply $B_I$ is perfectoid (see \cite[Remark~7.1.3]{SW-Berkeley}), which is the reason why we need to take such a cover.}
	By \cite[Corollary~II.1.12]{FS2024-1}, we can also assume each $R$ is a principal ideal domain. Thus by  \cite[Corollary~2.9.3]{Kedlaya-AWS}, each $R$ is a perfectoid field, hence $Y_{C^{\tilt},I}$ has finitely many rank-1-points. On the other hand, each $\rho\in [a,b]$ gives rise to a Gauss point $y_{\rho}$ of $Y_{C^{\tilt},I}$, which corresponds to a Gauss norm of a perfectoid open unit disk (see the proof of \cite[Lemma~II.1.14]{FS2024-1}). As such Gauss norms have rank 1 (see \cite[Example~2.3.7]{SW-Berkeley}), we get a contradiction, whence the proposition. 
\end{proof}

We now consider the case of higher $k$-Grassmannians. First we observe the following. 
\begin{lem}\label{lem-extn}
Assume $0<k<n$ and $kd\geq n$. For any $S\in \Perf_C$, there exists an exact  sequence of vector bundles over the relative Fargues--Fontaine curve $X_{S^{\tilt}}$: 
	\[
	0\rightarrow \OO(1/k)\rightarrow \cE_{n,d}\rightarrow \cE_{n-k, d-1}\rightarrow 0
\]
\end{lem}
\begin{proof}
	Let us first assume $S=\Spa(C,C^{+})$. 
	Assume $kd>n$, then  we get $(d-1)/(n-k)>d/n> 1/k$. Thus  the slope of $\OO(1/k)$ is smaller than that of $\cE_{n-k,d-1}$ and the straight Harder--Narasimhan polygon of $\cE_{n,d}$ lies below that of $\OO(1/k)\oplus \cE_{n-k,d-1}$. We then obtain the exact sequence by \cite[Theorem~1.1.2]{Hansenetal2022-1}. If $kd=n$, then all three vector bundles have slope $1/k$, hence we  have $\cE_{n,d}\cong \OO(1/k)^{\oplus d}$ and $\cE_{n-k,d-1}\cong \OO(1/k)^{\oplus (d-1)}$. The splitting of $\cE_{n,d}$ then gives the exact sequence. For general $S\in \Perf_C$, the claim follows by pulling back the exact sequence over $X_{C^{\tilt}}$ along $X_{S^{\tilt}}\rightarrow X_{C^{\tilt}}$. 
\end{proof}

Since the $k$-Grassmannian of $\bB^{\varphi^{n}=\pi^{d}}$ parametrizes injections $\cF\hookrightarrow \cE_{n,d}$ with $\cF$ fibrewise isomorphic to $\OO^{k}$, a short exact sequence as above allows us to consider such injections which factor through $\OO(1/k)$. 
\begin{defn}
	For any $k\geq 1$, let $H_{k}:=\Inj(\OO^{k}, \OO(1/k))/\underline{\GL_k(K)}$ and let $H_{k,C}:=H_k\times_{\Spd \Fq} \Spd C$. 
\end{defn}

\begin{lem}\label{lem-Zar2}
	When $kd\geq n$, there  exists a Zariski closed immersion $H_{k,C}\rightarrow \Gr_{k,C}(\bB^{\varphi^{n}=\pi^{d}})$. 
\end{lem}
\begin{proof}
 Using Lemma \ref{lem-extn}, we fix a short exact sequence 
	\begin{equation}\label{eq-Zar2-1}
		0\rightarrow \OO(1/k)\xrightarrow{\beta}\cE_{n,d}\xrightarrow{\gamma}\cE_{n-k,d-1}\rightarrow 0
	\end{equation}
	For any affinoid perfectoid $S=\Spa(R,R^{+})\in \Perf_{C}$, we want to show the fibre product 
	\[
	\begin{tikzcd}
		S\times_{\Gr_{k,C}(\bB^{\varphi^{n}=\pi^{d}})} H_{k,C}\arrow[r]\arrow[d] & S\arrow[d]\\
		H_{k,C}\arrow[r] & \Gr_{k,C}(\bB^{\varphi^{n}=\pi^{d}})
	\end{tikzcd}
	\] 
	is Zariski closed in $S$. As an $S$-point of the Grassmannian is a stably injective map $i: \cF\rightarrow \cE_{n,d}$ over $X_{S^{\tilt}}$ with $\cF\cong \OO^{k}$ at any geometric point, the fiber product parametrizes such $i$ that  factors through $\OO(1/k)$, or equivalently that $(\gamma\circ i)=0$. 

	Let $I=[1,q]\subset(0,\infty)$ and denote by  $\pi: Y_{S^{\tilt},I}\rightarrow X_{S^{\tilt}}$ the covering map. Let $B_{S^{\tilt},I}$ be the global sections of $Y_{S^{\tilt},I}$.  Via pullback of (\ref{eq-Zar2-1}) along $\pi$, we get a map 
	\[
		\pi^{\ast}(\gamma \circ i): F\rightarrow N
	\] 
	where $F=\Gamma(Y_{S^{\tilt},I}, \pi^{\ast}\cF)$ and $N=\Gamma(Y_{S^{\tilt},I}, \pi^{\ast}\cE_{n-k,d-1})$. By \cite[Theorem~5.2.8, Proposition~6.3.4]{SW-Berkeley}, both $F$ and $N$ are finite projective $B_{S^{\tilt},I}$-modules. The map $\pi^{\ast}(\gamma \circ i)$ then induces a $B_{S^{\tilt},I}$-linear  evaluation map $F\otimes _{B_{S^{\tilt},I}}N^{\ast}\rightarrow B_{S^{\tilt},I}$, hence its image defines an ideal of $B_{S^{\tilt},I}$ which we denote by $J$.  By Lemma \ref{lem-FS}, $J$ gives a  Zariski closed perfectoid subspace $Z_0\hookrightarrow  S^{\tilt}$ which is universal for perfectoid $T_0\rightarrow S^{\tilt}$ such that $J\OO(Y_{T_0,I})=0$. By the tilting equivalence $\Perf_{S}\cong \Perf_{S^{\tilt}}$, such a $Z_0\hookrightarrow  S^{\tilt}$ corresponds to a $Z\hookrightarrow S$, which, by \cite[Proposition~II.0.2]{FS2024-1}, is  again a Zariski closed perfectoid subspace. Thus $Z\hookrightarrow S$ is universal for perfectoid $T\rightarrow S$ such that $J\OO(Y_{T^{\tilt},I})=0$. 
	Assume $T$ is affinoid perfectoid, since $F$ and $N$ are finite projective, $J\OO(Y_{T^{\tilt},I})=0$ if and only if the map 
	\[
		(\pi^{\ast}(\gamma \circ i))_{T^{\tilt}}: F_{T^{\tilt}}=F\otimes_{B_{S^{\tilt},I}}B_{T^{\tilt},I}\rightarrow N_{T^{\tilt}}=N\otimes_{B_{S^{\tilt},I}}B_{T^{\tilt},I}
	\] 
	is zero. Moreover, since $I=[1,q]$, Frobenius translates of $Y_{T^{\tilt},I}$ cover $Y_{T^{\tilt}}$, hence the map above is zero if and only if 
	\[
		(\gamma \circ i)_{T^{\tilt}}: \cF|_{X_{T^{\tilt}}}\rightarrow \cE_{n-k,d-1}\Big|_{X_{T^{\tilt}}}
	\] 
	is the zero map. Thus $Z\subseteq S$ represents $S\times _{\Gr_{k,C}(\bB^{\varphi^{n}=\pi^{d}})} H_{k,C}$, whence the claim. 
\end{proof}

Therefore, we are reduced to showing  such $H_{k,C}$ are not perfectoid. Let us also make the following observation. 

\begin{lem}\label{lem-proj}
There exists a Zariski closed immersion $\P^{k-1,\diamond}\rightarrow H_{k,C}$. 
\end{lem}
\begin{proof}
	Let $\cM^{\LT}_{\infty}$ be the infinite-level Lubin--Tate space over $C$. Then $\P^{k-1,\diamond}$ is identified with $\cM^{\LT,\diamond}_{\infty}/\underline{\GL_k(K)}$. For $S\in \Perf_{C}$, this quotient  parametrizes injections $\cF\hookrightarrow \OO(1/k)$ such that $\cF$ is isomorphic to $\OO^{k}$ at all geometric points and the cokernel is isomorphic to $\iota_{\ast}W$ for some rank 1 projective  $\OO_{S}$-module $W$ \cite{SW2013-1}. Since the canonical divisor $\iota: S\rightarrow X_{S^{\tilt}}$ gives a section $\infty: \Spd C \rightarrow \Div^{1}_{C}$ and the Pl\"{u}cker morphism gives a map $H_{k,C}\rightarrow \Div^{1}_{C}$, we see that the fibre product
	\[
	\begin{tikzcd}
		H_{k,C}\times _{\Div^{1}_{C},\infty} \Spd C \arrow[d]\arrow[r] & H_{k,C}\arrow[d]\\
		\Spd C\arrow[r, "\infty"] & \Div^{1}_{C}
	\end{tikzcd}
	\] 
is exactly $\P^{k-1,\diamond}$. As $\infty$ is a Zariski closed immersion, so is the map $\P^{k-1,\diamond}\rightarrow H_{k,C}$.  
\end{proof}

\begin{prop}\label{prop-nonperf}
	The Grassmannian $\Gr_{k,C}(\bB^{\varphi^{n}=\pi^{d}})$ is not perfectoid if $kd\geq n$.  
\end{prop}
\begin{proof}
	By Lemma \ref{lem-Zar2}, it suffices to show $H_{k,C}$ is not perfectoid. When $k=1$, $H_{k,C}$ is the same as $\Div^{1}_{C}$ and Proposition \ref{prop-div} gives the result. When $k\geq 2$,  since $\P^{k-1,\diamond}$ is not perfectoid, Lemma \ref{lem-proj} implies that $H_{k,C}$ is not perfectoid. 
\end{proof}

\section{Mod $p$ Jacquet--Langlands correspondence}\label{section-6}
In this section, we apply Theorem \ref{thm-dk<n} to obtain vanishing results of Scholze's mod $p$ Jacquet--Langlands functor. Specifically, Theorem \ref{thm-dk<n} proves the conjecture of  Hansen--Mann, from which the applications  naturally follow  using tools developed in \cite{HM2022-1}. 

Recall from section \ref{section-1.3},  for $0<k<n$ and $S\in \Perf_{C}$, the moduli space  $W_{n,k}$ parametrizes surjections $\OO(1/n)\surj \cE$ over $X_{S^{\tilt}}$ with $\cE$ isomorphic to $\OO(1/k)$ at every  geometric point. Specifically, it is isomorphic to the Grassmannian $\Gr_{n-k,C}(\bB^{\varphi^{n}=\pi})$, see (\ref{eq-W-B}). We then deduce: 

\begin{prop}[Conjecture \ref{conj-HM}]\label{prop-d=1}
The space $W_{n,k}$ is perfectoid for all $0<k<n$. 
\end{prop}
\begin{proof}
	This follows directly from the $d=1$ case of Theorem \ref{thm-dk<n} and the identification recalled above. 
\end{proof}

\subsection{Scholze's functor}
We recall Scholze's functor for the mod $p$ Jacquet--Langlands correspondence and also Hansen--Mann's partial K\"{u}nneth formula.

From the infinite-level Lubin--Tate space $\cM^{\LT}_{\infty}$ over $C$, quotienting by $\GL_n(K)$ gives the projective space $\P^{n-1}$.  Let $\Rep^{\adm}_{\Fp}(\GL_n(K))$ be the category of admissible (smooth) representations of $\GL_n(K)$ over $\Fp$. For $\rho\in \Rep^{\adm}_{\Fp}(\GL_n(K))$, Scholze \cite{Scholze2018-1} constructed an \'etale sheaf $\cF_{\rho}$ over $\P^{n-1}$ which assigns  to $U\in (\P^{n-1})_{\et}$
\[
	\cF_{\rho}(U):=\Hom_{\cts, \GL_n(K)} (\left|U\times_{\P^{n-1}}\cM^{\LT}_{\infty}\right|, \rho).
\] 
This sheaf $\cF_{\rho}$ is $D^{\times}_{1/n}$-equivariant, and Scholze's functor is defined as 
\[
	\cJ^{i}_{n}(\rho):=H^{i}_{\et}(\P^{n-1}, \cF_{\rho}),
\] 
which gives an admissible  $D^{\times}_{1/n}$-representation. Specifically, $\cJ^{i}_{n}$ has the following vanishing bound.

\begin{thm}[{\cite[Theorem~1.1]{Scholze2018-1}}]\label{thm-Sch}
	For every  $\rho\in \Rep^{\adm}_{\Fp}(\GL_n(K))$, $\cJ^{i}_{n}(\rho)$ vanishes for  $i>2n-2$.
\end{thm}

Consider the quotient stack $[\P^{n-1}/D^{\times}_{1/n}]$. The commuting action of $\GL_n(K)$ and $D^{\times}_{1/n}$ on $\cM^{\LT}_{\infty}$ gives a $\GL_n(K)$-torsor and a $D^{\times}_{1/n}$-torsor over the stack. Write  $f_n: [\P^{n-1}/D^{\times}_{1/n}]\rightarrow B\underline{\GL_n(K)}$ and $g_n: [\P^{n-1}/D^{\times}_{1/n}]\rightarrow B\underline{D^{\times}_{1/n}}$ for the maps coming from these torsors, where both $f_n$ and $g_n$ are smooth, and $g_n$ is proper.  In \cite{HM2022-1}, Hansen-Mann considered the derived version of Scholze's functor
\[
	\cJ_{n}: D(\GL_n(K))\rightarrow D(D^{\times}_{1/n}),
\] 
which is defined as $\cJ_n=(g_n)_{\et,\ast}(f_n)^{\ast}_{\et}$.

For the space $W_{n,k}$, trivializing the target bundle $\OO(1/k)$ in its definition gives a  $D^{\times}_{1/k}$-torsor over it; trivializing the bundle $\OO^{n-k}$ from the identification (\ref{eq-W-inj}) gives a $\GL_{n-k}(K)$-torsor. We then get maps 
\[
\begin{tikzcd}
	& 	W_{n,k}\arrow[dl, "h_1",swap]\arrow[dr, "h_2"] & \\
	B\underline{\GL_{n-k}(K)} & & B\underline{D^{\times}_{1/k}}
\end{tikzcd}
\] 
For $0<k<n$, let $P_k(K)\subset \GL_n(K)$ be the standard maximal parabolic subgroup of block size $(n-k,k)$. Hansen-Mann obtained the following partial K\"{u}nneth formula:

\begin{thm}[{\cite[Theorem~5.2]{HM2022-1}}]\label{thm-HM}
	For $\sigma_1\in D(\GL_{n-k}(K))$ and $\sigma_2 \in D(\GL_k(K))$, we have 
	\[
		\cJ_{n}(\Ind^{\GL_n(K)}_{P_k(K)}(\sigma _1\boxtimes \sigma _2))\cong R\Gamma(W_{n,k}, h_{1,\et}^{\ast}\sigma _1\otimes h^{\ast}_{2,\et}\cJ_{k}(\sigma _2)).
	\] 
\end{thm}

\subsection{Vanishing bounds of Scholze's functor}
As one can readily see from \cite{JLH2021-1} and \cite{HM2022-1}, the perfectoidness of such moduli space of morphisms between vector bundles can be used  to prove  vanishing bounds of Scholze's functor sharper than the one in Theorem \ref{thm-Sch} on certain mod $p$ representations of $\GL_n(K)$. We now discuss the application of  Proposition \ref{prop-d=1}. 

As already dictated by  \cite[Proposition~5.5]{HM2022-1},  the first application of the perfectoidness of $W_{n,k}$ is the following: 
\begin{thm}\label{thm-para-ind}
	Let $0<k<n$. For any $\rho\in \Rep^{\adm}_{\Fp}(P_k(K))$, we have 
	\[
		\cJ^{i}_{n}(\Ind^{\GL_n(K)}_{P_k(K)}(\rho))=0 \quad \text{for all $i>n+k-2$}.
	\] 
\end{thm}
\begin{proof}
	The case when $\rho=\rho_1\boxtimes \rho_2$ with $\rho_1\in \Rep^{\adm}_{\Fp}(\GL_{n-k}(K))$ and $\rho_2\in \Rep^{\adm}_{\Fp}(\GL_k(K))$ is proved in \cite[Proposition~5.5]{HM2022-1}. We first briefly recall its proof. 

As $\cJ_n^{i}$ gives an $\Fp$-vector space, it is enough to prove the cohomology group in question, after tensoring with $\OO_{C}/p$,  is almost zero when $i>n+k-2$. We then have 
	\begin{align}
		&\cJ^{i}_{n}(\Ind^{\GL_n(K)}_{P_{n-k,k}}(\rho _1\boxtimes \rho _2))\otimes \OO_{C}/p\nonumber \\
		&\cong^{a}\cJ^{i}_{n}(\Ind^{\GL_n(K)}_{P_{n-k,k}}(\rho _1\boxtimes \rho_2)\otimes \OO^{+}/p)\nonumber\\
		&\cong H^{i}(W_{n,k}, h^{\ast}_{1,\et}\rho _1\otimes h^{\ast}_{2,\et}\cJ_k(\rho _2)\otimes \OO^{+}/p)\label{eq-HM-1}
	\end{align}
	where the first isomorphism is given by \cite[Theorem~3.2]{Scholze2018-1} and the second one is given by Theorem \ref{thm-HM}. By Theorem \ref{thm-Sch}, the cohomology of  $\cJ_k(\rho _2)$ vanishes in degrees larger than $2k-2$. The perfectoidness of the Grassmannian combined with \cite [Theorem~4.5]{Scheiderer1992-1} also implies the cohomology of the Grassmannian almost vanishes for degrees larger than its dimension; in this case it is $n-k$. Thus using the hypercohomology spectral sequence,  (\ref{eq-HM-1}) almost vanishes for $i>n+k-2$.

	To extend the result to general $\rho\in \Rep^{\adm}_{\Fp}(P_{k}(K))$, consider the projection  
\[
	s: B \underline{P_k(K)}\rightarrow  B\underline{\GL_{n-k}(K)}\times B\underline{\GL_{k}(K)}.
\] 
	Since admissible representations have trivial unipotent action, we have $\rho=s^{\ast}\ol{\rho}$ for some admissible representation $\ol{\rho}$ of $\GL_{n-k}(K)\times \GL_k(K)$. The same argument as in the proof of  \cite[Theorem~5.2, Lemma~5.3]{HM2022-1} then gives a natural isomorphism
\begin{equation*}
	\cJ_{n}(\Ind^{\GL_n(K)}_{P_k}(s^{\ast} \ol{\rho}) ) \simeq R\Gamma(W_{n,k}, (h_1,h_2)_{\et}^{\ast}(\id\times \cJ_k)(\ol{\rho})),
\end{equation*}
	where $\id\times \cJ_k=(\id\times g_k)_{\et,\ast}(\id\times f_k)^{\ast}_{\et}$. 
	Since $\id\times g_k$ is proper with $\P^{k-1}$-fibres,   Theorem \ref{thm-Sch} and proper base change imply   $(\id\times \cJ_k)(\ol{\rho})$ is concentrated in degrees $\leq 2k-2$. Using the perfectoidness of $W_{n,k}$ from Proposition \ref{prop-d=1} and the almost comparison with $\OO^{+}/p$-sheaves, the same argument as above gives the claimed vanishing bound. 
\end{proof}

\begin{rems}\label{rem-para-sharp}~{}
\begin{arenumerate}
\item In \cite[Theorem~5.3.1]{JLH2021-1}, Johansson--Ludwig  showed  that given a representation  $\rho\in \Rep^{\adm}_{\Fp}(P_1(K))$, 
	\[
		\cJ^{i}_{n}(\Ind^{\GL_n(K)}_{P_1}(\rho))=0 \quad \text{for all $i>n-1$}.
	\] 
		So the theorem above gives  a generalization of their result to all maximal parabolic subgroups of $\GL_n(K)$. Also notice that for $P_1(K)$, this bound is exactly half of Scholze's bound in Theorem  \ref{thm-Sch}. However, the  example in the following remark shows  that one should not  expect such  vanishing beyond the middle degree $n-1$ for all $P_k(K)\subset \GL_n(K)$. 
	\item Consider $\rho=\Ind^{\GL_n(K)}_{P_k}(\mathbf{1})$ for the trivial representation $\mathbf{1}$. Theorem \ref{thm-HM} gives 
	\[
		\cJ_{n}(\rho)\simeq R\Gamma(W_{n,k}, h^{\ast}_{2,\et} \cJ_{k}(\mathbf{1})).
	\] 
	Since $\cF_{\mathbf{1}}$ is the constant sheaf on $\P^{k-1}$, we get $\cJ_{k}(\mathbf{1})=R\Gamma (\P^{k-1},\Fp)$ which has highest nonzero degree $2k-2$. Therefore, we see $\cJ^{2k-2}_n(\rho)\neq 0$. For  $k>(n+1)/2$, this shows $\cJ^{i}_{n}$ can be nonvanishing beyond the middle degree  for representations induced from $P_k(K)$.
\end{arenumerate}
\end{rems}

At the moment, we do not know whether the bound given in Theorem \ref{thm-para-ind} is sharp, and it is natural to ask the following question
\begin{quest}\label{quest-sharp}
	For $0<k<n$, are there representations $\rho\in \Rep^{\adm}_{\Fp}(P_k(K))$ such that $\cJ^{n+k-2}_{n}(\Ind^{\GL_n(K)}_{P_k}(\rho))\neq 0$?
\end{quest}

For any standard parabolic subgroup $P\subseteq \GL_n$, we let 
\[
	\Sp_{P}= \frac{\Ind^{\GL_n(K)}_{P(K)}\mathbf{1}}{\sum^{}_{P\subsetneq Q}\Ind^{\GL_n(K)}_{Q(K)}\mathbf{1}}
\] 
be its associated \textit{generalized Steinberg representation}. Notice that when $P=\GL_n$, $\Sp_P$ gives the trivial representation of $\GL_n(K)$.

\begin{prop}\label{prop-Steinberg}
For each fixed $0<k<n$ and each $i>n+k-2$, there is a $D^{\times}_{1/n}$-equivariant isomorphism 
\[
	\cJ^{i}_{n}(\Sp_{P_k})\cong  \begin{cases}
		\Fp(- (i+1)/2)  &  \text{if $i$ is odd and $i\leq 2n-3$}\\
	0  &  \text{otherwise.}
	\end{cases}
\] 
\end{prop}
\begin{proof}
Since $P_k$ is maximal, we have a short exact sequence 
	\begin{equation}\label{eq-St}
	0\rightarrow \mathbf{1}\rightarrow \Ind^{\GL_n(K)}_{P_k(K)}(\mathbf{1})\rightarrow \Sp_{P_k}\rightarrow 0.
	\end{equation}
	Then in the  long exact sequence obtained from applying $\cJ_{n}$, the cohomology groups for the middle term vanish in  degrees  $>n+k-2$ by Theorem \ref{thm-para-ind}. Thus we have 
\[
	\cJ^{i}_{n}(\Sp_{P_k})\cong \cJ^{i+1}_{n}(\mathbf{1})=H^{i+1}_{\et}(\P^{n-1}, \Fp).
\] 
	Since  $H^{2j}_{\et}(\P^{n-1}, \Fp(j))\cong \Fp$ for every $0\leq j\leq n-1$, the claim follows. 
\end{proof}

Using the classification of irreducible admissible representations of $\GL_n(K)$ from \cite{HV2019-1} (also cf. \cite{Herzig2011-1} and \cite{AHHV2017-1}), we get the following vanishing result in the top degree allowed by Theorem \ref{thm-Sch}: 

\begin{thm}\label{thm-Herzig}
	Let $n>1$ and let $\rho\in \Rep^{\adm}_{\Fp}(\GL_n(K))$. Suppose that $\rho$ is irreducible, non-supersingular and infinite-dimensional. Then we have  $\cJ^{2n-2}_{n}(\rho)=0$. 
\end{thm}
\begin{proof}
	For any parabolic subgroup $P\subseteq \GL_n$, let $P=MN$ be its Levi decomposition. For an irreducible supersingular $\sigma \in \Rep^{\adm}_{\Fp}(M(K))$, we let $P^{\sigma }=M^{\sigma }N^{\sigma }$ be the maximal parabolic subgroup of $\GL_n$ containing $P$ such that $\sigma $ extends trivially on the unipotent subgroup. By \cite[Theorem~4]{HV2019-1}, for any irreducible $\rho\in \Rep^{\adm}_{\Fp}(\GL_n(K))$, there exists
	\begin{itemize}
		\item a parabolic subgroup $P=MN\subseteq \GL_n$ 
		\item a supersingular irreducible  $\sigma \in \Rep^{\adm}_{\Fp}(M(K))$ 
		\item a parabolic subgroup $Q\subseteq \GL_n$ with $P\subseteq Q\subseteq P^{\sigma }$
	\end{itemize}
	such that 
	\begin{equation}\label{eq-Herzig-1}
		\rho\cong \Ind^{\GL_n(K)}_{P^{\sigma }(K)}(\Sp^{M^{\sigma }(K)}_{Q(K)}(\sigma ))
	\end{equation}
	where
	\[\Sp^{M^{\sigma }(K)}_{Q(K)}(\sigma )=\frac{\Ind^{M^{\sigma }(K)}_{Q(K)}(\sigma )}{\sum^{}_{Q\subsetneq Q'\subseteq P^{\sigma }} \Ind^{M^{\sigma }(K)}_{Q'(K)} (\sigma )}
	\] 
	Since $\Sp^{M^{\sigma }(K)}_{Q(K)}(\sigma )$ admits a surjection from $\Ind^{M^{\sigma }(K)}_{Q(K)}(\sigma )$, (\ref{eq-Herzig-1}) gives a surjection 
	\begin{equation}\label{eq-surj-MFV}
		\Ind^{\GL_n(K)}_{Q(K)}(\sigma )\surj \rho.
	\end{equation}
	We claim $Q\neq \GL_n$. If $P=\GL_n$, then $\rho$ will be supersingular. If  $P\neq \GL_n$ and $Q=\GL_n$, then $\rho$ is identified with the extension of $\sigma $ to $\GL_n(K)$. Thus $N(K)$ acts trivially on $\rho$. Its normal closure is $\SL_n(K)$, so $\rho$ factors through $K^{\times}$ and gives an irreducible admissible representation of $K^{\times}$, which is finite-dimensional. This gives a contradiction. Then let $P_k$ be a  maximal parabolic subgroup containing $Q$. By Theorem \ref{thm-para-ind}, we get $\cJ^{2n-2}_{n}(\Ind^{\GL_n(K)}_{P_{k}(K)}(\sigma ))=0$, and right-exactness of $\cJ^{2n-2}_{n}$ gives $\cJ^{2n-2}_{n}(\rho)=0$. 
\end{proof}

Thus for irreducible non-supersingular infinite-dimensional $\rho\in \Rep^{\adm}_{\Fp}(\GL_n(K))$, we get an improved vanishing bound as $\cJ^{i}_{n}(\rho)=0$ for $i>2n-3$. 
\begin{rem}\label{rem-Herzig-sharp}
	The bound $2n-3$ is also sharp for such representations. Consider the short exact sequence (\ref{eq-St})  for $P_1(K)$. Applying the functor $\cJ_n$ to it, we get a long exact sequence. By Theorem \ref{thm-para-ind} and Proposition \ref{prop-Steinberg}, we get 
	\[
		\cJ^{2n-3}_{n}(\Sp_{P_1})\surj \cJ^{2n-2}_{n}(\mathbf{1})\neq 0.
	\] 
	Since Steinberg representations are non-supersingular and infinite-dimensional,  this gives examples of $\rho$ with nonvanishing $\cJ^{2n-3}_n(\rho)$. 
\end{rem}

\begin{rem}[Vign\'eras]\label{rem-MFV}
	The assumption of $\rho$ being infinite-dimensional is crucial in the proof  above. If $\rho\in \Rep^{\adm}_{\Fp}(\GL_n(K))$ is finite-dimensional, then $\rho$ factors through its determinant. 
	If $\rho$ is irreducible, it is then obtained by restriction of scalars from an admissible character $\chi\circ \det: \GL_n(K)\rightarrow \F^{\times}$ for some finite extension $\F/\Fp$. For a proper parabolic subgroup $P=MN$, let $\ol{P}$ denote its opposite parabolic subgroup and let $\Ord_{\ol{P}}$ denote Emerton's ordinary parts functor (\cite[\S3.1]{Emerton2010-1}).  Since the unipotent radical acts trivially on $\rho$, the Hecke operator defining $\Ord_{\ol{P}}$ is 0 in characteristic $p$, hence we have $\Ord_{\ol{P}}(\rho)=0$. As $\Ord_{\ol{P}}$ is right adjoint to $\Ind^{\GL_n(K)}_{P(K)}$ (\cite[Theorem~4.4.6]{Emerton2010-1}), we get $\Hom_{\GL_n(K)}(\Ind^{\GL_n(K)}_{P(K)}(\sigma ), \rho)=0$ for any $\sigma \in \Rep^{\adm}_{\Fp}(P(K))$. In particular, $\rho$ cannot be a quotient of a proper parabolic induction as in (\ref{eq-surj-MFV}).

	In the meantime,  finite-dimensional non-supersingular representations also give  exceptions to the vanishing bound in Theorem \ref{thm-Herzig}. For such a $\rho$, since  $\rho$ factors through the determinant, $\cF_{\rho}$ is constant on $\P^{n-1}$, cf. proof of \cite[Proposition~4.7]{Scholze2018-1}. Hence we have $\dim_{\Fp}\cJ^{2n-2}_n(\rho)=\dim_{\Fp}\rho>0$. 
\end{rem}

\printbibliography
\end{document}